\documentclass[12pt]{article}
\usepackage{amsmath,amsthm,amssymb}
\usepackage[hang]{caption2}
\usepackage{indentfirst}
\usepackage[section]{placeins}
\numberwithin{equation}{section}     
\usepackage{graphics}
\usepackage{pdfsync}
\usepackage{hyperref}

\def\accentsfrancais{applemac}
  \RequirePackage[\accentsfrancais]{inputenc}

\newtheorem{thm}{Theorem}
\newtheorem{theorem}[thm]{Theorem}
\numberwithin{thm}{section}
\newtheorem{lemma}[thm]{Lemma}
\newtheorem{cor}[thm]{Corollary}

\newtheorem{remark}[thm]{Remark}

\def\ds{\displaystyle}
\def\emptyset{/\kern-.51em o}
\def\eq{\mathop{\vrule height2,6pt depth-2,3pt
         width -1pt\kern 0pt =}}

\let\norbali\normalbaselines
\def\anorbali{\norbali\advance\lineskip\jot
\advance\baselineskip\jot\advance\lineskiplimit\jot}
\def\ouvre{\let\normalbaselines\anorbali}
\def\D{{\mathbb{D}}}

\def\D{{\mathbb{D}}}
\def\R{{\mathbb{R}}}

\def\P{{\mathbb{P}}}

\def\N{\rm \hbox{I\kern-.2em\hbox{N}}}
\def\Z{\rm \hbox{Z\kern-.3em\hbox{Z}}}

\def\Ge{{\bf e}}

\def\Gr{{\bf r}}

\def\Gu{{\bf u}}
\def\Gv{{\bf v}}

\def\GE{{\bf E}}
\def\GF{{\bf F}}
\def\GG{{\bf G}}

\def\GI{{\bf I}}

\def\GQ{{\bf Q}}
\def\GR{{\bf R}}

\def\GF{{\bf F}}

\def\GX{{\bf X}}

\def\eps{\varepsilon}

\begin{document}
\title{Asymptotic behavior of a structure made by a  plate and a straight rod.}

\author{D. Blanchard$^{1}$, G. Griso$^{2}$}
\date{}
\maketitle
 
{\footnotesize
\begin{center}

$^{1}$  Universit\'e de Rouen, UMR 6085,  76801   Saint Etienne du Rouvray Cedex, France, \\ E-mail: dominique.blanchard@univ-rouen.fr

$^{2}$ Laboratoire J.-L. Lions--CNRS, Bo\^\i te courrier 187, Universit\'e  Pierre et
Marie Curie,\\ 4~place Jussieu, 75005 Paris, France, \; Email: griso@ann.jussieu.fr\\

\end{center} }

\begin{abstract}
 This paper  is devoted  to describe the asymptotic behavior of a structure made by a  thin plate and a thin rod in the framework of nonlinear elasticity. We scale the applied forces in such a way that the  level of the total elastic energy leads to the Von-K\'arm\'an's equations (or the linear model for smaller forces) in the plate and to a one dimensional rod-model at the limit. The junction conditions include in particular the continuity of the bending in the plate and the stretching in the rod at the junction. 
\end{abstract}
 


\section{ Introduction} 
\medskip
In this paper we consider the junction problem between a plate and a rod as their thicknesses  tend to  zero. We denote by $\delta$ and $\eps$ the respective half thickness of the plate $\Omega_\delta$ and the rod $B_\eps$. The structure is clamped on a part of the lateral boundary of the plate and it is free on the rest of its boundary. We assume that this multi-structure is made of  elastic materials (possibly different in the plate and in the rod). In order to simplify the analysis we consider Saint-Venant-Kirchhoff's materials with Lam's coefficients of order 1 in the plate and of order $q_\eps^2=\eps^\eta$ in the rod with $\eta>-1$ (see (1.1)). It allows us to deal with a rod made of the same material as the plate, or made of a softer material ($\eta>0$) or of a stiffer material ($-1<\eta<0$). It is well known that the limit behaviors in both the two parts of this multi-structure depend on the order of the infimum  of the elastic energy with respect to the parameters $\delta$ and  $\eps$.  Indeed this order is governed by the ones of the applied forces on the structure.   In the present paper, we suppose  that the orders of the applied forces  depend on $\delta$ (for the plate) and $\eps$ (for the rod) and   via two new real parameters $\kappa$ and $\kappa^{'}$ (see Subsection 5.1). The parameters $\kappa$, $\kappa^{'}$ and $\eta$ are linked in such a way that the infimum of the total elastic energy be of order $\delta^{2\kappa-1}$. As far as a minimizing sequence $v_\delta$  of the energy is concerned, this    leads to the following estimates of the Green-St Venant's strain tensors
\begin{equation*}
 \big\|\nabla v^T_\delta\nabla v_\delta-\GI_3\big\|_{L^2(\Omega_{\delta} ; \R^{3\times 3})}\le  C\delta^{\kappa-1/2},\quad  \big\|\nabla v^T_\delta\nabla v_\delta-\GI_3\big\|_{L^2(B_{\eps} ; \R^{3\times 3})}\le C{\delta^{\kappa-1/2}\over q_\eps}.
 \end{equation*}  The limit model for the plate is the Von K\'arm\'an system ($\kappa=3$) or the classical linear plate model $(\kappa>3$). Similarly, in order to obtain either a nonlinear model or the classical linear model in the rod, the order of $\big\|\nabla v^T_\delta\nabla v_\delta-\GI_3\big\|_{L^2(B_{\eps} ; \R^{3\times 3})}$ must be less than $\eps^{\kappa^{'}}$  with $\kappa^{'}\ge 3$. Hence, $\delta$, $\eps$ and $q_\eps$ are linked by the relation
 $$\delta^{\kappa-1/2}= q_\eps\eps^{\kappa^{'}}.$$
Moreover, still for the above  estimates of the Green-St Venant's strain tensors, the bending in the plate is of order $\delta^{\kappa-2}$ and the stretching in the rod is of order $\eps^{\kappa^{'}-1}$. Since, we wish at least these two quantities to match at the junction it is essential to have
 $$\delta^{\kappa-2}= \eps^{\kappa^{'}-1}.$$ Finally, the two relations between the parameters  lead to 
\begin{equation}\label{DEQ}
\delta^3=q^2_\eps\eps^2=\eps^{2+2\eta}.
\end{equation} Under the relation \eqref{DEQ}, we prove that in the limit model, the rotation of the cross-section and the bending of the rod in the junction are null. The limit  plate model (nonlinear or linear) is coupled with the limit rod model (nonlinear or linear) via the bending in the plate and the stretching in the rod.

A similar problem, but starting within the framework of the linear elasticity is investigated in \cite {Murat}. In this work the rod is also clamped at its bottom. This additional boundary condition makes easier the analysis of the linear  system of elasticity. In \cite {Murat}, the authors also assume that
\begin{equation}\label{CondSup}
{\eps\over \delta^2}\longrightarrow +\infty.
\end{equation} With this extra condition they obtain the same linear limit model as we do here in the case $\kappa>3$ and $\kappa^{'}>3$ and they wonder if the condition \eqref{CondSup} is necessary or purely technical in order to obtain the junction conditions. The present article shows that this condition is not necessary to carry out the analysis.  

The derivation of the limit behavior of a multi-structure such as the one considered here rely on two main arguments. Firstly it is convenient to derive   ''Korn's type inequalities'' both in the plate and the rod.  Secondly one needs estimates of a deformation in the junction (in order  to obtain the limit junction conditions). In this paper this is achieved through  the use of   two main tools given  in Lemmas \ref{lemme2}  and \ref{lemAp}. For the plate, since it is clamped on a part of its lateral boundary, a 'Korn's type inequality'' is  given in \cite{BGJE}. For the rod the issue is more intricate because the rod is nowhere clamped. In a first step, we derive sharp estimates of a deformation $v$ in the junction with respect to the parameters and to the $L^2$ norm (over the whole structure) of the linearized strain tensor $\nabla v+(\nabla v)^T-2\GI_3$. This is the object of Lemma \ref{lemme2}. In a second step, in Lemma \ref{lemAp},  we  estimate the $L^2$ norm of  the linearized strain tensor of $v$ in the rod  with respect to the parameters and to the $L^2$ norms of $\hbox{dist}(\nabla v , SO(3))$ in the rod and in the plate. The proofs of these two lemmas strongly rely on the decomposition techniques for the displacements and the deformations of the plate and the rod. Once these technical results are established, we are in a position to scale the applied forces and in the case $\kappa=3$ or $\kappa^{'}=3$ to state an adequate assumption on these forces in order to finally obtain a total elastic energy of order less than $\delta^5$.

In Section 2 we introduce a few general notations. Section 3 is devoted to recall a main tool that we use in the whole paper, namely  the decomposition technique of the deformation of thin structures. In Section 4, the estimates provided by this method allow us  to derive sharp estimates on the bending and the cross-section rotation of the rod at the junction together  with the difference between the bending of the plate and the stretching of the rod at the junction. In Section 5 we introduce the elastic energy and we precise the scaling with respect to $\delta$ and $\kappa$  on the applied forces in order to obtain a total elastic energy of order $\delta^{2\kappa-1}$. In Section 6 we give the asymptotic behavior of the Green-St-Venant's strain tensors in the plate and in the rod. In Section 7 we characterize the limit of the sequence of the rescaled infimum of the elastic energy in terms of the minimum of a limit energy.

As general references on the theory of elasticity we refer to \cite{Ant1} and \cite{C1}. The reader is referred to \cite{Ant}, \cite{Trab}, \cite{GROD} for an introduction of rods models and to \cite{C11}, \cite{C2}, \cite{Ciarlet3}, \cite{FJM} for plate models. As far as junction problems in multi-structures we refer to \cite{CDN}, \cite{C2}, \cite{DJR}, \cite{DJP}, \cite{Ldret},  \cite{Auf}, \cite{Gru1}, \cite{Gru2}, \cite{GSP}, \cite{Murat}, \cite{BGG1}, \cite{BGG2}, \cite{BG1}, \cite{GSR}, \cite{BGNOTE}. For the decomposition method in thin structures we refer to \cite{GDecomp}, \cite{BGRod}, \cite{BGJE}, \cite{SimplCoq}.

\vskip 1mm
\section { Notations and definition of the structure.}
\vskip 1mm
Let us introduce a few  notations and definitions concerning the geometry of the plate and the rod. We denote  $I_d$  the identity map of $\R^3$. 
 
Let $\omega$ be a bounded domain in $\R^2$ with lipschitzian boundary included in the plane $(O; \Ge_1, \Ge_2)$  such that $O\in \omega$ and let $\delta>0$. 
The plate is the domain
 $$ \Omega_\delta=\omega\times ]-\delta,\delta[.$$
Let $\gamma_0$ be an open subset of $\partial \omega$ which is made of a finite number of connected components (whose closure are disjoint). The corresponding lateral part of the boundary of $\Omega_\delta$ is
$$\Gamma_{0,\delta}=\gamma_0\times]-\delta,\delta[.$$ The rod is defined by  
$$B_{\varepsilon,\delta}=D_\eps \times ]-\delta,L[,\qquad D_\eps=D(O,\varepsilon),\qquad D=D(O,1)$$ where $\varepsilon>0$ and where $D_r=D(O,r)$ is the disc of radius $r$ and center the origin $O$. The whole structure is denoted
 $${\cal S}_{\delta,\varepsilon}=\Omega_\delta\cup B_{\varepsilon,\delta}$$
 while the junction is 
 $$C_{\delta,\varepsilon}=\Omega_\delta\cap B_{\varepsilon,\delta}=D_\eps \times ]-\delta,\delta[.$$
   The set of admissible deformations of the plate is
$$\D_\delta=\Big\{ v\in H^1(\Omega_\delta ; \R^3)\;\;|\;\; v=I_d\enskip\hbox{on}\enskip\Gamma_{0,\delta}\Big\}.$$  The set of admissible deformations of the structure is
$$\D_{\delta,\varepsilon}=\Big\{ v\in H^1({\cal S}_{\delta,\varepsilon} ; \R^3)\;\;|\;\; v=I_d\enskip\hbox{on}\enskip\Gamma_{0,\delta}\Big\}.$$
The aim of this paper is to study the asymptotic behavior of the structure ${\cal S}_{\delta,\eps}$ in the case where the both paremeters $\delta$ and $\eps$ go to 0. In order to simplify this study, we link $\delta$ and $\eps$ by assuming that
\begin{equation}\label{Assump}
\hbox{there exists $\theta\in \R^*_+$ such that }\quad \delta=\eps^\theta
\end{equation} where $\theta$ is a fixed constant (see Subsection 5.1). Nevertheless,   we keep the parameters $\delta$ and $\eps$ in the estimates given in Sections 3 and 4.
\vskip 1mm
\section { Some recalls about the decompositions in the plates and the rods.}
\vskip 1mm
From now on, in order to simplify the notations, for any open set ${\cal O}\subset \R^3$ and any field $u\in H^1({\cal O} ; \R^3)$, we denote by
$$\GG_s(u,{\cal O})=||\nabla u+(\nabla u)^T||_{L^2({\cal O};\R^{3\times 3})}.$$

We recall  Theorem 4.3  established in \cite{GDecomp}. Any displacement $u\in   H^1(\Omega_\delta; \R^3)$ of the plate is decomposed as 
\begin{equation}\label{FDec}
u(x)={\cal U} (x_1,x_2)+ x_3{\cal R}(x_1,x_2) \land \Ge_3+\overline{u} (x),\qquad x\in \Omega_\delta
\end{equation}
where ${\cal U}$ and ${\cal R}$ belong to $ H^1(\omega; \R^3)$ and $\overline{u} $ belongs to $  H^1(\Omega_\delta; \R^3)$.  The sum of the two first terms $U_e(x)={\cal U} (x_1,x_2)+x_3{\cal R}(x_1,x_2)\land \Ge_3$ is called the elementary displacement associated to $u$.

\noindent  The following Theorem is proved in \cite{GDecomp}.
\begin{theorem}\label{Theorem 3.3.}
Let $u\in   H^1(\Omega_\delta; \R^3)$, there exists  an elementary displacement $U_e(x)={\cal U}(x_1,x_2)+x_3{\cal R}(x_1,x_2)\land \Ge_3$ and a warping  $\overline{u}$  satisfying \eqref{FDec} such that 
\begin{equation}\label{3.7}
\begin{aligned}
&||\overline{u} ||_{L^2(\Omega_\delta; \R^3)}\le C\delta \GG_s(u,\Omega_\delta),\quad||\nabla \overline{u} ||_{ L^2(\Omega_\delta; \R^3)}\le C \GG_s(u,\Omega_\delta)\\
&\Bigl\|{\partial {\cal R}\over \partial x_\alpha}\Big\|_{ L^2(\omega; \R^3)}\le {C\over \delta^{3/2}} \GG_s(u,\Omega_\delta)\\
& \Bigl\|{\partial{\cal U}\over \partial x_\alpha}-{\cal R}\land \Ge_\alpha\Big\|_{ L^2(\omega; \R^3)}\le {C\over \delta^{1/2}}\GG_s(u,\Omega_\delta)
\end{aligned}
\end{equation} where the constant $C$ does not depend on $\delta$.
\end{theorem}
 The warping $\overline{u}$ satisfies the following relations
 \begin{equation}\label{RelWarPlaque}
 \begin{aligned}
& \int_{-\delta}^\delta\overline{u}(x_1,x_2,x_3)dx_3=0,\qquad  \int_{-\delta}^\delta x_3\overline{u}_\alpha(x_1,x_2,x_3)dx_3=0\quad \hbox{for a.e. } (x_1,x_2)\in \omega.
\end{aligned} \end{equation}
If a deformation  $v$ belongs to $\D_\delta$ then  the displacement $u=v-I_d$ is equal to $0$ on  $\Gamma_{0,\delta}$. In this case  the the fields ${\cal U}$, ${\cal R}$  and the warping $\overline{u}$ satisfy
\begin{equation}\label{CLUR}
{\cal U}= {\cal R}=0\qquad \hbox{on } \enskip \gamma_0,\qquad \overline{u}=0\qquad \hbox{on}\quad \Gamma_{0,\delta}.
\end{equation}
 Then,  from   \eqref{3.7}, for any deformation $v\in \D_\delta$   the corresponding displacement $u=v-I_d$ verifies the following estimates (see also \cite{GSP}):
\begin{equation}\label{Estm}
\begin{aligned}
||{\cal R}||_{H^1(\omega; \R^3)}+||{\cal U}_3||_{H^1(\omega)}\le {C\over \delta^{3/2}} \GG_s(u,\Omega_\delta),\\
 ||{\cal R}_3||_{L^2(\omega)}+||{\cal U}_\alpha||_{H^1(\omega)}\le {C\over \delta^{1/2}} \GG_s(u,\Omega_\delta).
\end{aligned} \end{equation} The constants depend only on $\omega$.

From the above estimates we deduce the following Korn's type inequalities for the displacement $u$
\begin{equation}\label{KoP0}
\begin{aligned}
&||u_\alpha||_{L^2(\Omega_\delta)}\le C\GG_s(u,\Omega_\delta),\quad ||u_3||_{L^2(\Omega_\delta)}\le {C\over \delta}\GG_s(u,\Omega_\delta),\\
& ||u-{\cal U}||_{L^2(\Omega_\delta ; \R^3)}\le {C\over \delta}\GG_s(u,\Omega_\delta),\\
& ||\nabla u||_{L^2(\Omega_\delta;\R^9)}\le {C\over \delta}\GG_s(u,\Omega_\delta).
 \end{aligned}
\end{equation}

Now, we consider a displacement  $u\in H^1(B_{\varepsilon,\delta} ; \R^3)$ of the rod $B_{\varepsilon,\delta}$. This displacement can be   decomposed as (see Theorem 3.1 of \cite{GDecomp})
\begin{equation}\label{DecR}
u(x)={\cal W} (x_3)+{\cal Q}(x_3)\land\big(x_1\Ge_1+x_2\Ge_2\big)+\overline{w}(x),\qquad x\in B_{\varepsilon,\delta},
 \end{equation}
where ${\cal W} $, ${\cal Q}$ belong to $H^1(-\delta,L;\R^3)$ and $\overline{w}$ belongs to $ H^1(B_{\varepsilon,\delta};\R^3)$. 
The sum of the  two first terms ${\cal W} (x_3)+{\cal Q}(x_3)\land\big(x_1\Ge_1+x_2\Ge_2\big)$ is called an elementary displacement of the rod.

The following Theorem is established in \cite{GDecomp} (see Theorem 3.1).
\smallskip
\begin{theorem}\label{Theorem II.2.2.}  Let $u\in H^1(B_{\varepsilon,\delta};\R^3)$, there exists an elementary displacement ${\cal W}(x_3) +{\cal Q}(x_3)\land \big(x_1\Ge_1 +x_2\Ge_2 \big)$ and a warping $\overline{w}$ satisfying \eqref{DecR} and such that 
\begin{equation}\label{EstmRod}
\begin{aligned}
&||\overline{w}||_{L^2(B_{\varepsilon,\delta};\R^3)}\le C\varepsilon\GG_s(u,B_{\eps,_\delta}),\quad||\nabla\overline{w}||_{L^2(B_{\varepsilon,\delta};\R^{3\times 3})}\le C \GG_s(u,B_{\eps,_\delta})\\
&\Bigl\|{d{\cal Q}\over dx_3}\Big\|_{L^2(-\delta, L ;\R^{3})}\le {C\over\varepsilon^2} \GG_s(u,B_{\eps,_\delta})\\
& \Bigl\|{d{\cal W}\over dx_3}-{\cal Q} \land\Ge_3\Big\|_{L^2(-\delta,L;\R^3)}\le {C\over \varepsilon}\GG_s(u,B_{\eps,_\delta})
\end{aligned}
\end{equation}
 where the constant $C$ does not depend on $\varepsilon$, $\delta$ and $L$.
 \end{theorem}
 The warping $\overline{w}$ satisfies the following relations
 \begin{equation}\label{RelWarPoutre}
 \begin{aligned}
& \int_{D_\eps}\overline{w}(x_1,x_2,x_3)dx_1dx_2=0,\qquad  \int_{D_\eps}x_\alpha\overline{w}_3(x_1,x_2,x_3)dx_1dx_2=0,\\
& \int_{D_\eps}\big\{x_1\overline{w}_2(x_1,x_2,x_3)-x_2\overline{w}_1(x_1,x_2,x_3)\big\}dx_1dx_2=0\quad \hbox{for a.e. } x_3\in ]-\delta,L[.
\end{aligned} \end{equation}
Then, from   \eqref{EstmRod}, for any displacement $u\in H^1(B_{\eps,\delta};\R^3)$  the terms of the decomposition of $u$ verify
\begin{equation}\label{EstmR}
\begin{aligned}
||{\cal Q}-{\cal Q}(0)||_{H^1(-\delta,L;\R^3)}\le {C\over \eps^2} \GG_s(u,B_{\eps,_\delta}),\\
||{\cal W}_3-{\cal W}_3(0)||_{H^1(-\delta,L)}\le {C\over \eps} \GG_s(u,B_{\eps,_\delta}),\\
||{\cal W}_\alpha-{\cal W}_\alpha(0)||_{H^1(-\delta,L)}\le {C\over \eps^2} \GG_s(u,B_{\eps,_\delta})+C\eps||{\cal Q}(0)||_2.
\end{aligned} \end{equation}
Now, in order to obtain  Korn's type inequalities for the displacement $w$, the following section is devoted to  give estimates on ${\cal Q}(0)$ and ${\cal W}(0)$.

\section{Estimates at the junction.}
\vskip 1mm

Let us set 
$$H^1_{\gamma_0}(\omega)=\{\varphi\in H^1(\omega);\; \varphi=0\; \hbox  { on } \gamma_0\}.$$
Let $v\in \D_{\delta,\eps}$ be a deformation whose displacement $u=v-I_d$ is decomposed as in Theorem \ref{Theorem 3.3.} and Theorem \ref{Theorem II.2.2.}.
We define the function $\widetilde{\cal U}_3$ as the solution of the following variational problem
\begin{equation}\label{defVtilde}
\left\{\begin{aligned}
& \widetilde{\cal U}_3\in H^1_{\gamma_0}(\omega),\\
&\int_\omega\nabla\widetilde{\cal U}_3\nabla\varphi=\int_\omega\big({\cal R}\land\Ge_\alpha\big)\cdot\Ge_3 {\partial\varphi\over \partial x_\alpha},\\
&\forall\varphi\in H^1_{\gamma_0}(\omega).
\end{aligned}\right.
\end{equation}
Indeed $\widetilde{\cal U}_3$  satisfies due to the third estimate in \eqref{Estm}
\begin{equation}\label{estVtilde}
||\widetilde{\cal U}_3||_{H^1(\omega)}\le {C\over \delta^{3/2}}\GG_s(u,\Omega_\delta)
\end{equation}
The definition \eqref{defVtilde} of $\widetilde{\cal U}_3$ together with the fourth estimate in \eqref{3.7} lead to 
\begin{equation}\label{V-Vtilde}
||{\cal U}_3-\widetilde{\cal U}_3||_{H^1(\omega)}\le {C\over \delta^{1/2}}\GG_s(u,\Omega_\delta)
\end{equation} and moreover
\begin{equation}\label{NablaV-R}
\begin{aligned}
&\Big\|{\partial\widetilde{\cal U}_3\over \partial x_\alpha}-({\cal R}\land \Ge_\alpha)\cdot\Ge_3\Big\|_{L^2(\omega)}\le {C\over \delta^{1/2}}\GG_s(u,\Omega_\delta).
\end{aligned}\end{equation}
Now, let $\rho_0>0$ be fixed such that $D(O,\rho_0)\subset\subset  \omega$. Since ${\cal R}\in H^1(\omega ; \R^3)$, the function $\widetilde{\cal U}_3$  belongs to $H^2\big(D(O,\rho_0)\big)$ and the third estimate in \eqref{Estm} gives
\begin{equation}\label{estVtildeH2}
||\widetilde{\cal U}_3||_{H^2(D(O,\rho_0))}\le {C\over \delta^{3/2}}\GG_s(u,\Omega_\delta).
\end{equation} Hence $\widetilde{\cal U}_3$ belongs to ${\cal C}^0(\overline{D(O,\rho_0)})$.

\begin{lemma}\label{lemme2} We have the following estimates on ${\cal W}(0)$:
\begin{equation}\label{W1(0)}
|{\cal W}_\alpha(0)|^2\le {C\over\eps\delta}\big[\GG_s(u,\Omega_\delta)\big]^2+ C\Big[1+{\delta^2\over\eps^2}\Big]{\delta\over \eps^2}\big[\GG_s(u,B_{\eps,\delta})\big]^2
\end{equation}
and
\begin{equation}\label{W3(0)-V3(0,0)}
|{\cal W}_3(0)-\widetilde {\cal U}_3(0,0)|^2\le {C\over\delta^2}\Big[1+{\eps^2\over \delta}\Big]\big[\GG_s(u,\Omega_\delta)\big]^2+C{\delta\over \eps^2 }\big[\GG_s(u,B_{\eps,\delta})\big]^2.
\end{equation}
The vector  ${\cal Q}(0)$ satisfies the following estimate:
\begin{equation}\label{Q(0)First}
||{\cal Q}(0)||^2_2\le {C\over\eps^2\delta}\Big[1+{\eps\over \delta^2}\Big]\big[\GG_s(u,\Omega_\delta)\big]^2+C{\delta\over\eps^4}\big[\GG_s(u,B_{\eps,\delta})\big]^2.
\end{equation}
The constants $C$ are independent of $\eps$ and $\delta$.
\end{lemma}

\begin{proof} The two decompositions of $u=v-I_d$ give, for a.e. $x$ in the common part of the plate and the rod $C_{\delta,\eps}$
\begin{equation}\label{DecCyl}
\begin{aligned}
{\cal U} (x_1,x_2)+x_3{\cal R}(x_1,x_2)\land\Ge_3+\overline{u}(x)={\cal W} (x_3)+{\cal Q}(x_3)\land(x_1\Ge_1+x_2\Ge_2)+\overline{w}(x).
\end{aligned}\end{equation} 
{\it Step 1. Estimates on ${\cal W}(0)$.}

\noindent In this step we prove \eqref{W1(0)} and \eqref{W3(0)-V3(0,0)}.
Taking into account the equalities \eqref{RelWarPlaque} and \eqref{RelWarPoutre} on the warpings $\overline{u}$ and $\overline{w}$, we deduce that the averages on the cylinder $C_{\delta,\eps}$ of the both sides of the above equality \eqref{DecCyl} give
\begin{equation}\label{V-WFirst}
{\cal M}_{D_{\eps}}\big({\cal U}\big)={\cal M}_{I_{\delta}}\big({\cal W}\big)
\end{equation} where $\ds {\cal M}_{D_{\eps}}\big({\cal U}\big)={1\over |D_\eps|}\int_{D_\eps}{\cal U}(x_1,x_2)dx_1dx_2$ and
$\ds {\cal M}_{I_{\delta}}\big({\cal W}\big)={1\over 2\delta}\int_{-\delta}^\delta{\cal W}(x_3)dx_3$. 

\noindent Besides using \eqref{Estm} we have
\begin{equation*}
||{\cal U}_\alpha||^2_{L^2(D_\eps )} \le C\eps||{\cal U}_\alpha||^2_{L^4(\omega)}\le C\eps ||{\cal U}_\alpha||^2_{H^1(\omega)}\le{C\eps\over \delta}\big[\GG_s(u,\Omega_\delta)\big]^2.
\end{equation*} From these estimates we get
\begin{equation}\label{Walpha}
|{\cal M}_{I_{\delta}}\big({\cal W}_\alpha\big)|^2=|{\cal M}_{D_{\eps}}\big({\cal U}_\alpha\big)|^2\le {C\over\eps \delta}\big[\GG_s(u,\Omega_\delta)\big]^2.
\end{equation} Moreover, for any $p\in [2,+\infty[$ using \eqref{V-Vtilde} we deduce that
\begin{equation}\label{DefU3}
\begin{aligned}
||{\cal U}_3-\widetilde{\cal U}_3||_{L^2(D_\eps)} & \le C\eps^{1-2/p}||{\cal U}_3-\widetilde{\cal U}_3||_{L^p(\omega)}\\
& \le C_p\eps^{1-2/p} ||{\cal U}_3-\widetilde{\cal U}_3||_{H^1(\omega)}\le C_p{\eps^{1-2/p}\over \delta^{1/2}}\GG_s(u,\Omega_\delta).
\end{aligned}\end{equation} 
Then we replace ${\cal U}_3$ with $\widetilde{\cal U}_3$ in \eqref{V-WFirst} to obtain
\begin{equation}\label{V-WSec}
|{\cal M}_{D_{\eps}}\big(\widetilde{\cal U}_3\big)-{\cal M}_{I_{\delta}}\big({\cal W}_3\big)|^2\le{C_p\over \eps^{4/p}\delta}\big[\GG_s(u,\Omega_\delta)\big]^2.
\end{equation}
We carry on by comparing ${\cal M}_{D_{\eps}}\big(\widetilde{\cal U}_3\big)$ with $\widetilde{\cal U}_3(0,0)$.  Let us set
\begin{equation}\label{Ralpha}
{\Gr}_\alpha={\cal M}_{D_\eps}\big({\cal R}\land\Ge_\alpha\big)\cdot\Ge_3\big)={1\over |D_\eps|}\int_{D_\eps }\big({\cal R}(x_1,x_2)\land\Ge_\alpha\big)\cdot\Ge_3dx_1dx_2
\end{equation} and consider the function $\Psi(x_1,x_2)=\widetilde{\cal U}_3(x_1,x_2)-{\cal M}_{D_{\eps}}\big(\widetilde{\cal U}_3\big)-x_1\Gr_2-x_2\Gr_1$. Due to the  estimate \eqref{estVtildeH2} we first obtain
\begin{equation}\label{Psi2}
\Big\|{\partial^2\Psi\over \partial x_\alpha\partial x_\beta}\Big\|_{L^2(D_\eps)}\le {C\over \delta^{3/2}}\GG_s(u,\Omega_\delta).
\end{equation} Secondly, from \eqref{3.7} and the Poincar-Wirtinger's inequality in the disc $D_\eps $ we get
\begin{equation*}
||({\cal R}\land\Ge_\alpha)\cdot\Ge_3-{\cal M}_{D_\eps}\big(({\cal R}\land\Ge_\alpha)\cdot\Ge_3\big)||_{L^2(D_\eps)}\le C{\eps\over \delta^{3/2}}\GG_s(u,\Omega_\delta).
\end{equation*} Using the above inequality and \eqref{NablaV-R} we deduce that
\begin{equation}\label{NablaPsi}
||\nabla\Psi||^2_{L^2(D_\eps ; \R^2)}\le C \Big({1\over \delta}+{\eps^2\over \delta^3}\Big)\big[\GG_s(u,\Omega_\delta)\big]^2,\\
\end{equation} Noting that ${\cal M}_{D_\eps}(\Psi)=0$,  the above inequality and the Poincar-Wirtinger's inequality in the disc $D_\eps $ and   lead to
\begin{equation}\label{Psi}
||\Psi||^2_{L^2(D_\eps)}\le C{\eps^2\over\delta}\Big(1+{\eps^2\over \delta^2}\Big)\big[\GG_s(u,\Omega_\delta)\big]^2.
\end{equation} From inequalities  \eqref{Psi2}, \eqref{NablaPsi} and \eqref{Psi} we deduce that
\begin{equation*}
||\Psi||^2_{{\cal C}^0(\overline{D_\eps})}\le C  \Big({1\over \delta}+{\eps^2\over \delta^3}\Big)\big[\GG_s(u,\Omega_\delta)\big]^2
\end{equation*} which in turn gives
\begin{equation*}
|\Psi(0,0)|^2= |\widetilde{\cal U}_3(0,0)-{\cal M}_{D_{\eps}}\big(\widetilde{\cal U}_3\big)|^2\le C \Big({1\over \delta}+{\eps^2\over \delta^3}\Big)\big[\GG_s(u,\Omega_\delta)\big]^2.
\end{equation*} This last estimate and \eqref{V-WSec} yield 
\begin{equation}\label{V(0)-WSec}
| \widetilde{\cal U}_3(0,0)-{\cal M}_{I_{\delta}}\big({\cal W}_3\big)|^2\le{C\over \delta}\Big({C_p\over\eps^{4/p}}+{\eps^2\over \delta^2}\Big)\big[\GG_s(u,\Omega_\delta)\big]^2.
\end{equation} In order to estimate ${\cal M}_{I_{\delta}}\big({\cal W}_3\big)-{\cal W}_3(0)$, we set $\ds y(x_3)={\cal W}(x_3)-{\cal Q}(0)x_3\land\Ge_3$. Estimates in Theorem   \ref{Theorem II.2.2.} together with the use of Poincar inequality in order to estimate $||{\cal Q}-{\cal Q}(0)||_{L^2(-\delta,\delta ; \R^3)}$ give
\begin{equation*}
\begin{aligned}
&\Bigl\|{dy_\alpha \over dx_3}\Big\|_{L^2(-\delta, \delta)}\le C\Big({1\over \eps}+{\delta\over \eps^2}\Big)\GG_s(u,B_{\eps,\delta}),\\
&\Bigl\|{dy_3 \over dx_3}\Big\|_{L^2(-\delta, \delta)}\le {C\over \eps}\GG_s(u,B_{\eps,\delta}).
\end{aligned}\end{equation*}  which imply
\begin{equation*}
\begin{aligned}
&\bigl\|y_\alpha-y_\alpha(0)\big\|^2_{L^2(-\delta,\delta)}\le C{\delta^2\over \eps^2}\Big(1+{\delta^2\over \eps^2}\Big)\big[\GG_s(u,B_{\eps,\delta})\big]^2,\\
&\bigl\|y_3-y_3(0)\big\|^2_{L^2(-\delta,\delta)}\le C{\delta^2\over \eps^2}\big[\GG_s(u,B_{\eps,\delta})\big]^2.
\end{aligned}\end{equation*} Then, taking the averages on $]-\delta, \delta[$ we obtain
\begin{equation}\label{y-y(0)}
\begin{aligned}
&|{\cal M}_{I_{\delta}}\big({\cal W}_\alpha\big)-{\cal W}_\alpha(0)|^2
\le  C\Big(1+{\delta^2\over \eps^2}\Big){\delta\over \eps^2}\big[\GG_s(u,B_{\eps,\delta})\big]^2,\\
&|{\cal M}_{I_{\delta}}\big({\cal W}_3\big)-{\cal W}_3(0)|^2
\le  C{\delta\over \eps^2}\big[\GG_s(u,B_{\eps,\delta})\big]^2.
\end{aligned}\end{equation}
Finally, from \eqref{Walpha},  \eqref{V(0)-WSec} and  the above last inequality,  we obtain \eqref{W1(0)} and the following estimate:
\begin{equation}\label{W3(0)-V3(0,0)p}
|{\cal W}_3(0)-\widetilde {\cal U}_3(0,0)|^2\le {C\over\delta}\Big[{C_p\over\eps^{4/p}}+{\eps^2\over \delta^2}\Big]\big[\GG_s(u,\Omega_\delta)\big]^2)+C{\delta\over \eps^2 }\big[\GG_s(u,B_{\eps,\delta})\big]^2.
\end{equation} Choosing $p=\max(2, 4 / \theta)$ (recall that $\delta=\eps^\theta$) we get \eqref{W3(0)-V3(0,0)}.
\vskip 1mm
\noindent{\it Step 2. We prove the estimate on ${\cal Q}(0)$.} 
 We recall (see Definition 3 in \cite{GDecomp}) that the field ${\cal Q}$ is defined by 
\begin{equation*}
\begin{aligned}
& {\cal Q}_1(x_3)={4\over \pi\eps^4}\int_{D_\eps} x_1u_3(x) dx_1dx_2, \qquad {\cal Q}_2(x_3)=-{4\over \pi\eps^4}\int_{D_\eps}x_2u_3(x)dx_1dx_2,\\
& {\cal Q}_3(x_3)={2\over \pi\eps^4}\int_{D_\eps}\big\{x_1u_2(x)-x_2u_1(x)\big\}dx_1dx_2,\qquad \hbox{for a.e. } x_3\in ]-\delta,L[.
\end{aligned}\end{equation*}
Now, again using the equalities \eqref{RelWarPlaque} and \eqref{RelWarPoutre} on the warpings $\overline{u}$ and $\overline{w}$,  the two decompositions \eqref{DecCyl} of $u$ in the cylinder $C_{\delta,\eps}$ lead to
\begin{equation*}
\Big|{\eps^2\over 4}{\cal M}_{I_\delta}({\cal Q}_\alpha)\Big|=\Big|{\cal M}_{D_\eps}\big({\cal U}_3\,x_\alpha\big)\Big|,\qquad \Big|{\eps^2\over 2}{\cal M}_{I_\delta}({\cal Q}_3)\Big|=\Big|{\cal M}_{D_\eps}\big({\cal U}_2\,x_1-{\cal U}_1\,x_2\big)\Big|.
\end{equation*} Noticing that ${\cal M}_{D_\eps}\big({\cal U}_1\,x_2\big)={\cal M}_{D_\eps}\big([{\cal U}_1-{\cal M}_{D_\eps}({\cal U}_1)]x_2\big)$ and applying the Poincar-Wirtinger's inequality with \eqref{Estm} yield
\begin{equation}\label{S40}
\big|{\cal M}_{I_\delta}({\cal Q}_3)\big|^2\le {C\over \eps^2\delta}[\GG_s(u,\Omega_\delta)]^2.
\end{equation} From the definition of the function $\Psi$ and the constants $\Gr_\alpha$ introduced in Step 1 we deduce that
\begin{equation}\label{S400}
\big|{\cal M}_{D_\eps}\big({\cal U}_3x_\alpha\big)\big|\le \big|{\cal M}_{D_\eps}\big(\Psi x_\alpha\big)\big|+\big|{\cal M}_{D_\eps}\big([{\cal U}_3-\widetilde{U}_3]x_\alpha\big)\big|+C\eps^2|\Gr_\alpha|.
\end{equation} Estimate \eqref{Psi} give
\begin{equation}\label{S41}
\big|{\cal M}_{D_\eps}\big(\Psi x_\alpha\big)\big|^2\le C{\eps^2\over\delta}\Big(1+{\eps^2\over \delta^2}\Big)\big[\GG_s(u,\Omega_\delta)\big]^2
\end{equation} while  \eqref{Estm} leads to 
\begin{equation}\label{S42}
|\Gr_\alpha|^2\le {C\over \eps^2}||{\cal R}||^2_{L^2(D_\eps;\R^3)}\le {C\over \eps} ||{\cal R}||^2_{L^4(D_\eps;\R^3)}\le {C\over \eps} ||{\cal R}||^2_{H^1(\omega;\R^3)} \le {C\over \eps\delta^3}[\GG_s(u,\Omega_\delta)]^2
\end{equation} and \eqref{V-Vtilde} with the Poincar-Wirtinger's inequality yield
\begin{equation}\label{S43}
\big|{\cal M}_{D_\eps}\big([{\cal U}_3-\widetilde{U}_3]x_\alpha\big)\big|^2\le {C\eps^2\over \delta}[\GG_s(u,\Omega_\delta)]^2
\end{equation} Finally, \eqref{S400}, \eqref{S41}, \eqref{S42} and \eqref{S43} we obtain
\begin{equation}\label{S44}
\big|{\cal M}_{I_\delta}({\cal Q}_\alpha)\big|^2\le {C\over\eps^2\delta}\Big(1+{\eps\over \delta^2}\Big)\big[\GG_s(u,\Omega_\delta)\big]^2
\end{equation} The third estimate in \eqref{EstmRod} implies
\begin{equation}\label{S45}
\big\|{\cal Q}(0)-{\cal M}_{I_\delta}({\cal Q})\big\|^2_2 \le C{\delta\over \eps^4}[\GG_s(u,B_{\eps,\delta})]^2.
\end{equation} From \eqref{S44} and \eqref{S45} we get \eqref{Q(0)First}.
\end{proof}

\section{Elastic structure.}
  \medskip
In this section  we assume that  the  structure ${\cal S}_{\delta,\eps}$ is made of an elastic material. The associated local energy $\widehat{W}_\eps\; :\;  \GX_3\longrightarrow \R^+$ is the following  St Venant-Kirchhoff's law (see [9]) 
\begin{equation}\label{HatW2}
\widehat{W}_\eps(F)=\left\{\begin{aligned}
&Q_\eps(F^TF-\GI_3)\quad\hbox{if}\quad \det(F)>0\\
&+\infty\hskip 2cm \hbox{if}\quad \det(F)\le 0.
\end{aligned}\right.
\end{equation} where the quadratic form $Q$ is given by
\begin{equation}\label{HatW3}
Q_\eps(E)=\left\{
\begin{aligned}
&Q_p(E)\quad \hbox{in the plate}\; \Omega_\delta,\\
&q_\eps^2Q_r(E) \quad \hbox{in the rod}\; B_{\eps,\delta},
\end{aligned}\right.
 \end{equation} with
 \begin{equation}\label{QRQP}
Q_p(E)={\lambda_p\over 8}\big(tr(E) \big)^2+{\mu_p\over 4}tr\big(E^2\big),\qquad 
Q_r(E)={\lambda_r\over 8}\big(tr(E) \big)^2+{\mu_r\over 4}tr\big(E^2\big), \end{equation}
and  where $(\lambda_p,\mu_p)$ (resp.  $(q_\eps^2\lambda_r, q_\eps^2\mu_r)$) are the Lam's coefficients of the plate (resp. the rod). The constant $q_\eps$ depends only on the rod, we set $q_\eps=\eps^\eta$, the parameter $\eta$ being such that

$\bullet$ $\eta=0$  for the same order for the the Lam's coefficients in the plate and the rod,

$\bullet$ $\eta>0$  for a softer  material in the rod than in the plate,

$\bullet$ $\eta<0$  for a softer  material in the plate than in the rod.
\vskip 1mm
 Let us recall (see e.g. \cite{FJM} or \cite{BGRod}) that for any $3\times 3$ matrix $F$ such that $\det(F)>0$ we have
\begin{equation}\label{HatW4}
tr([F^TF-\GI_3]^2) = |||F^TF-\GI_3|||^2\ge \hbox { dist }(F,SO(3))^2.
\end{equation}
\vskip 1mm
Hence, we denote by
\begin{equation}\label{GE}
{\cal E}(u,{\cal S}_{\delta,\varepsilon})=[\GG_s(u, \Omega_\delta))]^2+q_\eps^2[\GG_s(u, B_{\eps,\delta})]^2
\end{equation}  the linearized energy of a displacement $u\in H^1({\cal S}_{\delta,\varepsilon};\R^3)$.  We define   the total energy  
$J_{\delta}(v)$\footnote{For later convenience, we have added the term $\ds \int_{{\cal S}_{\delta,\eps}}f_{\delta}(x)\cdot  I_d(x)dx$
to the usual standard energy, indeed this does not affect the  minimizing problem for  $J_{\delta}$. }  over $\D_{\delta,\eps}$ by
\begin{equation}\label{J}
J_{\delta}(v)=\int_{{\cal S}_{\delta,\eps}}\widehat{W}_\eps(\nabla  v)(x)dx-\int_{{\cal S}_{\delta,\eps}}f_{\delta}(x)\cdot (v(x)-I_d(x))dx.
\end{equation}
\vskip 1mm
\subsection{Relations between $\delta$, $\eps$ and $q_\eps$.}
 In Section Subsection 5.2 we scale the applied forces in order to have the infimum of this total energy of order $\delta^{2\kappa-1}$ with $\kappa\ge 3$. In such way, the minimizing sequences $(v_\delta)$ satisfy
\begin{equation*}
 \big\|\nabla v^T_\delta\nabla v_\delta-\GI_3\big\|_{L^2(\Omega_{\delta} ; \R^{3\times 3})}\le  C\delta^{\kappa-1/2},\quad  \big\|\nabla v^T_\delta\nabla v_\delta-\GI_3\big\|_{L^2(B_{\eps} ; \R^{3\times 3})}\le C{\delta^{\kappa-1/2}\over q_\eps}.
 \end{equation*}  The above estimate in the plate $\Omega_\delta$ leads to the Von K\'arm\'an limit model ($\kappa=3$) or the classical linear plate model ($\kappa>3$). Since we wish at least to recover the linear model in the rod which corresponds to a Green-St Venant's strain tensor in the rod of order $\eps^{\kappa^{'}}$ with $\kappa^{'}>3$, we are led to assume that
\begin{equation}\label{Lien1}
\delta^{\kappa-1/2}=q_\eps\eps^{\kappa^{'}}.
\end{equation} 
Furthermore,  still for the above  estimates of the Green-St Venant's strain tensors, the bending in the plate if of order $\delta^{\kappa-2}$ and the stretching in the rod is of order $\eps^{\kappa^{'}-1}$. In this paper,  we wish these two quantities to match at the junction it is essential to have
\begin{equation}\label{Lien2}
\delta^{\kappa-2}=\eps^{\kappa^{'}-1}.
\end{equation}

As a consequence of the above relations  \eqref{Lien1} and \eqref{Lien2} we deduce that
\begin{equation}\label{relETA}
\delta^3=q_\eps^2\eps^{2}=\eps^{2\eta+2}
\end{equation}  which implies that $\eta$ must be chosen such that $\eta>-1$.  

{\it From now on we assume that \eqref{relETA} holds true and to recover a slightly general model in the rod we extend the analysis to  $\kappa^{'}\ge 3$.}
\vskip 1mm
\subsection{Assumptions on the forces and energy estimate.}

Let $v\in \D_{\delta,\eps}$ be a deformation. The estimates in Lemma \ref{lemme2} become (taking into account \eqref{relETA})

\begin{equation}\label{Q(0)-W(0)}\begin{aligned}
&|{\cal W}_\alpha(0)|^2\le  {C\over \delta^2}\Big[1+{\delta^2\over \eps^2}\Big]{\cal E}(u,{\cal S}_{\delta,\varepsilon}),\\
&|{\cal W}_3(0)-\widetilde {\cal U}_3(0,0)|^2\le {C\over \delta^3} (\eps^2+\delta){\cal E}(u,{\cal S}_{\delta,\varepsilon})\\
&||{\cal Q}(0)||^2_2\le  {C\over \eps\delta^2}\Big[{1\over \delta}+{1\over \eps}\Big]{\cal E}(u,{\cal S}_{\delta,\eps})\le  C( \delta+ \eps){{\cal E}(u,{\cal S}_{\delta,\eps})\over \eps^2\delta^3}.
\end{aligned}\end{equation} 
The following lemma give the estimates of the displacement $u=v-I_d$ in the rod $B_{\eps,\delta}$.

\begin{lemma}\label{lemme41} For any deformation $v$ in $\D_{\delta,\eps}$ the displacement $u=v-I_d$ satisfies the following Korn's type inequality  in the rod $B_{\eps,\delta}$: 
\begin{equation}\label{KoP}
\begin{aligned}
&||u_\alpha||^2_{L^2(B_{\eps,\delta})}\le   C{{\cal E}(u,{\cal S}_{\delta,\varepsilon})\over \eps^2q_\eps^2},\quad ||u_3||^2_{L^2(B_{\eps,\delta})}\le  C{{\cal E}(u,{\cal S}_{\delta,\varepsilon})\over q_\eps^2},\\
& ||\nabla u||^2_{L^2(B_{\eps,\delta};\R^9)}\le   C{{\cal E}(u,{\cal S}_{\delta,\varepsilon})\over \eps^2q_\eps^2},\quad ||u-{\cal W}||^2_{L^2(B_{\eps,\delta};\R^3)}\le  C{{\cal E}(u,{\cal S}_{\delta,\varepsilon})\over q_\eps^2}.
 \end{aligned}
\end{equation} 
\end{lemma}

\begin{proof} We define  the rigid displacement $\Gr$  by $\Gr(x)={\cal W}(0)+{\cal Q}(0)\land x$. From  \eqref{EstmR} we obtain the following inequalities for the displacement $u-r$:
\begin{equation}\label{KoPr}
\begin{aligned}
&||u_\alpha-\Gr_\alpha||_{L^2(B_{\eps,\delta})}\le  {C\over \eps} \GG_s(u,B_{\eps,\delta}),\\
&||u_3-\Gr_3||_{L^2(B_{\eps,\delta})}\le C \GG_s(u,B_{\eps,\delta}),\\
& ||\nabla u-\nabla \Gr||_{L^2(B_{\eps,\delta};\R^9)}\le  {C\over \eps} \GG_s(u,B_{\eps,\delta}).
 \end{aligned}
\end{equation} Then, the above estimates and \eqref{Q(0)-W(0)} give (observe that  due to relation \eqref{relETA} we have $\ds ||{\cal Q}(0)||^2_2\le {C\over \eps^4 q_\eps^2}{\cal E}(u,{\cal S}_{\delta,\eps})$) 
\begin{equation*}
\begin{aligned}
&||\Gr_\alpha||^2_{L^2(B_{\eps,\delta};\R^3)}\le    {C\over \eps^2 q_\eps^2}{\cal E}(u,{\cal S}_{\delta,\eps}),\quad ||\Gr_3||^2_{L^2(B_{\eps,\delta};\R^3)}\le {C\over q_\eps^2}{\cal E}(u,{\cal S}_{\delta,\eps}),\\
& ||\nabla \Gr||^2_{L^2(B_{\eps,\delta};\R^9)}\le  {C\over \eps^2 q_\eps^2}{\cal E}(u,{\cal S}_{\delta,\eps}).
 \end{aligned}
\end{equation*} which lead to the first third estimates in \eqref{KoP} using  \eqref{KoPr}. Before obtaining the estimate of $u-{\cal W}$ we write (see \eqref{DecR})
\begin{equation*}
u(x)-{\cal W}(x_3)=\big({\cal Q}(x_3)-{\cal Q}(0)\big)\land(x_1\Ge_1+x_2\Ge_2)+\overline{u}(x)+{\cal Q}(0)\land(x_1\Ge_1+x_2\Ge_2).
\end{equation*} Then due to estimates \eqref{EstmRod}, \eqref{EstmR} and \eqref{Q(0)-W(0)} we finally get the last inequality in \eqref{KoP}.
\end{proof} 
The following lemma is one of the key point of this article in order to obtain a priori estimates on minimizing sequences of the total energy.
\vskip 1mm
\begin{lemma}\label{lemAp} Let $v\in \D_{\delta,\eps}$ be a deformation and $u=v-I_d$. We have
\begin{equation}\label{GsPlaque8}
\GG_s(u,\Omega_\delta)\le C||dist(\nabla v, SO(3))||_{L^2(\Omega_\delta)}+C_1{||dist(\nabla v, SO(3))||^2_{L^2(\Omega_\delta)}\over \delta^{5/2}}
\end{equation} and the following estimate on $\GG_s( u, B_{\eps,\delta})$:
\begin{equation}\label{GGs11}
\begin{aligned}
\GG_s( u, B_{\eps,\delta}) \le  C||\hbox{dist}(\nabla v , SO(3))||_{L^2(B_{\eps,\delta})}  & +C_2{||\hbox{dist}(\nabla v , SO(3))||^2_{L^2(B_{\eps,\delta})}\over \eps^3}\\
 &\hskip-1cm  + C\big[\delta+\eps^{1/2}\big]{||dist(\nabla v, SO(3))||^2_{L^2(\Omega_\delta)}\over \eps\delta^{3}}.
\end{aligned}\end{equation}  The constants $C$ do not depend on $\delta$ and $\eps$.
\end{lemma}
 The proof is postponed in the Appendix.
 \vskip 1mm
 As an immediate consequence of the Lemmas \ref{lemme41} and \ref{lemAp}, we get the full estimates of the displacement $u=v-I_d$ in the rod.
 \vskip 1mm
 \begin{cor} For any deformation $v$ in $\D_{\delta,\eps}$ the displacement $u=v-I_d$ satisfies the following nonlinear Korn's type inequality  in the rod $B_{\eps,\delta}$: 
\begin{equation}\label{KoP2}
\begin{aligned}
||u_\alpha||_{L^2(B_{\eps,\delta})} & \le   C\Big[{||dist(\nabla v, SO(3))||_{L^2(\Omega_\delta)}\over \eps q_\eps}
+(\sqrt\delta+\sqrt\eps){||dist(\nabla v, SO(3))||^2_{L^2(\Omega_\delta)}\over \eps^3q_\eps^2}\Big]\\
&+ C{||\hbox{dist}(\nabla v , SO(3))||_{L^2(B_{\eps,\delta})} \over \eps}+2C_2{||\hbox{dist}(\nabla v , SO(3))||^2_{L^2(B_{\eps,\delta})}\over \eps^4},\\
||u_3||_{L^2(B_{\eps,\delta})} & \le   C\Big[{||dist(\nabla v, SO(3))||_{L^2(\Omega_\delta)}\over q_\eps}
+(\sqrt\delta+\sqrt\eps){||dist(\nabla v, SO(3))||^2_{L^2(\Omega_\delta)}\over \eps^2q_\eps^2}\Big]\\
&+ C{||\hbox{dist}(\nabla v , SO(3))||^2_{L^2(B_{\eps,\delta})}}+2C_2{||\hbox{dist}(\nabla v , SO(3))||^2_{L^2(B_{\eps,\delta})}\over \eps^2},\\
||\nabla u||_{L^2(B_{\eps,\delta};\R^9)}& \le   C\Big[{||dist(\nabla v, SO(3))||_{L^2(\Omega_\delta)}\over \eps q_\eps}
+(\sqrt\delta+\sqrt\eps){||dist(\nabla v, SO(3))||^2_{L^2(\Omega_\delta)}\over \eps^3q_\eps^2}\Big]\\
&+ C{||\hbox{dist}(\nabla v , SO(3))||_{L^2(B_{\eps,\delta})} \over \eps}+2C_2{||\hbox{dist}(\nabla v , SO(3))||^2_{L^2(B_{\eps,\delta})}\over \eps^4}.
 \end{aligned}
\end{equation}

 \end{cor}
\noindent {\bf  First assumptions on the forces. }{\it  
To introduce the scaling on $f_{\delta}$, let us consider   $f_r$, $g_1$, $g_2$ in $ L^2(0,L; \R^3)$ and  $f_p\in L^2(\omega;\R^3)$ and assume  that the force  $f_{\delta}$ is given  by
\begin{equation}\label{ForceP}
\begin{aligned}
f_{\delta}(x)&=q_\eps^2 \eps^{\kappa^{'}} \Big[f_{r,1}(x_3)\Ge_1+f_{r,2}(x_3)\Ge_2+{1\over \eps}f_{r,3}(x_3)\Ge_3+{x_1\over \eps^2} g_{1}(x_3)+{x_2\over \eps^2}g_{2}(x_3)\Big],\\ 
&\quad x\in B_{\eps ,\delta},\quad x_3>\delta,\\ \\
f_{\delta,\alpha}(x)&=\delta^{\kappa-1}  f_{p,\alpha}(x_1,x_2),\qquad f_{\delta,3}(x)=\delta^{\kappa}  f_{p,3}(x_1,x_2),\qquad x\in \Omega_\delta.
\end{aligned}\end{equation} }
We set
\begin{equation}\label{Forces}
N(f_p)=||f_p||_{L^2(\omega;\R^3)},\qquad \qquad N(f_r)=||f_r||_{L^2(0,L;\R^3)}+\sum_{\alpha=1}^2||g_\alpha||_{L^2(0,L;\R^3)}.
\end{equation}
\begin{lemma}\label{LemCond} Let $v\in \D_{\delta,\eps}$  be such that $J(v)\le 0$ and $u=v-I_d$.  Under the assumption \eqref{ForceP} on the applied forces, we have

\noindent $\bullet$ if $\kappa>3$ and $\kappa^{'}>3$ then 
\begin{equation}\label{EstEnerg}
\begin{aligned}
||dist(\nabla v, SO(3))||_{L^2(\Omega_\delta)}+q_\eps||\hbox{dist}(\nabla v , SO(3)) & ||_{L^2(B_{\eps,\delta})}\\
\le C\delta^{\kappa-1/2}\big(N(f_p)+ & N(f_r)\big),
\end{aligned}\end{equation}
\noindent $\bullet$ if $\kappa=3$ and $\kappa^{'}>3$ then there exists a constant $C^*$  which do not depend on $\delta$ and $\eps$ such that, if the forces applied to the plate $\Omega_\delta$ satisfy
 \begin{equation}\label{Cst}
N(f_p) < {C^* \mu_p}
\end{equation} then \eqref{EstEnerg} still holds true,

\noindent $\bullet$ if $\kappa>3$ and $\kappa^{'}=3$ then there exists a constant $C^{**}$  which do not depend on $\delta$ and $\eps$ such that, if the forces applied to the rod $B_{\eps,\delta}$ satisfy
 \begin{equation}\label{Cstst}
N(f_r) < {C^{**} \mu_r}
\end{equation} then \eqref{EstEnerg} still holds true,

\noindent $\bullet$ if $\kappa=3$ and $\kappa^{'}=3$ then  if the applied forces satisfy \eqref{Cst} and \eqref{Cstst}  then \eqref{EstEnerg} still holds true.

\noindent The constants $C$, $C^*$ and $C^{**}$ depend only on $\omega$ and $L$.
\end{lemma}
\vskip 1mm
Recall that we want a geometric energy in the plate $||dist(\nabla v, SO(3))||_{L^2(\Omega_\delta)}$ of order less than $\delta^{5/2}$ in order to obtain a limit Von K\'arm\'an plate model. Lemma \ref{LemCond} prompts us to adopt the conditions \eqref{Cst} if $\kappa=3$ and \eqref{Cstst} if $\kappa^{'}=3$. Let us notice that in the case $\kappa=3$ under the only assumption \eqref{ForceP} on the forces (i.e. without assumption \eqref{Cst}) the geometric energy is generally of order $\delta^{3/2}$ which corresponds to a limit model allowing  large deformations  (see \cite{SimplCoq}).
\vskip 1mm
\noindent {\bf  Second assumptions on the forces. }{\it From now on, in the whole paper we assume that

\noindent $\bullet$ if $\kappa=3$ then 
 \begin{equation}\label{Cstar}
N(f_p) < {C^{*} \mu_p},
\end{equation}
$\bullet$ if $\kappa^{'}=3$ then 
\begin{equation}\label{Cstarstar}
N(f_r) < {C^{**} \mu_r}.
\end{equation}}
\vskip 1mm
\begin{proof}{\it Proof of Lemma \ref{LemCond}.}
 Notice that $J_{\delta}(I_d)=0$. So,  in order to minimize $J_{\delta}$ we only need to  consider  deformations $v $ of $\D_{\delta,\eps}$ such that   $J_{\delta}(v )\le 0$.
From \eqref{KoP0}, \eqref{KoP} and  the assumptions \eqref{ForceP} on the body forces, we obtain for any $v\in \D_{\delta,\eps}$ and for  $u=v-I_d$
\begin{equation}\label{nivf}
\begin{aligned}
\Big|\int_{{\cal S}_{\delta,\eps}}f_{\delta}(x)\cdot u(x)dx \Big| & \le C_3  \delta^{\kappa-1/2}N(f_p)\GG_s(u, \Omega_\delta)\\
& +C_4q_\eps\eps^{\kappa^{'}}N(f_r)\sqrt{{\cal E}(u,{\cal S}_{\delta,\eps})}.
\end{aligned}\end{equation} Now we  use  the definition \eqref{GE} ${\cal E}(u,{\cal S}_{\delta,\eps})$ and  Lemma \ref{lemAp} to bound $\GG_s(u, \Omega_\delta)$ and $\GG_s(u, B_{\eps,\delta})$ and ${\cal E}(u,{\cal S}_{\delta,\eps})$. Taking into account  the relations \eqref{Lien1}-\eqref{relETA} we obtain
\begin{equation}\label{nivf}
\begin{aligned}
\Big|\int_{{\cal S}_{\delta,\eps}}f_{\delta}(x)\cdot u(x)dx \Big|\le & C_1C_3\delta^{\kappa-3}N(f_p)||dist(\nabla v, SO(3))||^2_{L^2(\Omega_\delta)} \\
&\hskip-1.5cm  + C\big[\sqrt\delta+\sqrt \eps\big]\eps^{\kappa^{'}-3}N(f_r)||dist(\nabla v, SO(3))||^2_{L^2(\Omega_\delta)}\\
&\hskip-1.5cm  +2C_2C_4 q^2_\eps\eps^{\kappa^{'}-3}N(f_r)||\hbox{dist}(\nabla v , SO(3))||^2_{L^2(B_{\eps,\delta})}\\
&\hskip-1.5cm + C\delta^{\kappa-1/2}\big\{N(f_p)+N(f_r)\big\}||dist(\nabla v, SO(3))||_{L^2(\Omega_\delta)}\\
&\hskip -1.5cm +C q^2_\eps\eps^{\kappa^{'}}N(f_r)\big||\hbox{dist}(\nabla v , SO(3))||_{L^2(B_{\eps,\delta})}.
\end{aligned}\end{equation} From \eqref{HatW2}, \eqref{HatW3}, \eqref{QRQP} and \eqref{HatW4} we have 
\begin{equation}\label{coerW}
\begin{aligned}
{\mu_p\over 4}||dist(\nabla v, SO(3))||^2_{L^2(\Omega_\delta)}+{\mu_rq_\eps^2\over 4}||dist(\nabla v, SO(3))||^2_{L^2(B_{\eps,\delta})}\\\le \int_{{\cal S}_{\delta,\eps}}\widehat{W}_\eps(\nabla  v)(x)dx\le \int_{{\cal S}_{\delta,\eps}}f_{\delta}(x)\cdot u(x)dx.
\end{aligned}\end{equation} Then using \eqref{nivf} we get

\begin{equation}\label{GsPlaque}
\begin{aligned}
&\Big[{\mu_p\over 4}-C_1C_3\delta^{\kappa-3}N(f_p)-C\big[\delta+\eps^{1/2}\big]\eps^{\kappa^{'}-3}N(f_r)\Big]
||dist(\nabla v, SO(3))||^2_{L^2(\Omega_\delta)}\\
+&\Big[{\mu_r\over 4}-2C_2C_4 \eps^{\kappa^{'}-3}N(f_r)\Big]q_\eps^2||dist(\nabla v, SO(3))||^2_{L^2(B_{\eps,\delta})}\\
\le & C\delta^{\kappa-1/2}\big\{N(f_p)+N(f_r)\big\}||dist(\nabla v, SO(3))||_{L^2(\Omega_\delta)}\\
+ &C q^2_\eps\eps^{\kappa^{'}}N(f_r)||\hbox{dist}(\nabla v , SO(3))||_{L^2(B_{\eps,\delta})}\\
\le & C\delta^{\kappa-1/2}\big\{N(f_p)+N(f_r)\big\}\big(||dist(\nabla v, SO(3))||_{L^2(\Omega_\delta)}+q_\eps||\hbox{dist}(\nabla v , SO(3))||_{L^2(B_{\eps,\delta})}\big).
\end{aligned}\end{equation} Now, recall that $\kappa\ge 3$ and $\kappa^{'}\ge 3$, so that first $\big[\delta+\eps^{1/2}\big]\eps^{\kappa^{'}-3}\to 0$. Secondly, setting $C^*=4C_1C_3$ and $C^{**}=8C_2C_4$ then \eqref{EstEnerg} holds true in any case of the lemma.
\end{proof}
\noindent Recalling that $\delta^{\kappa-1/2}=q_\eps\eps^{\kappa^{'}}$, we first deduce from Lemma  \ref{LemCond}
\begin{equation}\label{EstDist}
||dist(\nabla v, SO(3))||_{L^2(\Omega_\delta)}\le C\delta^{\kappa-1/2},\quad  ||\hbox{dist}(\nabla v , SO(3)) ||_{L^2(B_{\eps,\delta})}\le C\eps^{\kappa^{'}}.\end{equation} Then applying \eqref{GsPlaque} of Lemma \ref{lemAp} we obtain
\begin{equation}\label{EstGsPlaque}
\GG_s(u,\Omega_\delta) \le C\delta^{\kappa-1/2}
\end{equation} while \eqref{GGs11} gives
\begin{equation*}
\GG_s(u,B_{\eps,\delta}) \le C\eps^{\kappa^{'}}+C\big[\delta+\eps^{1/2}\big]{||dist(\nabla v, SO(3))||_{L^2(\Omega_\delta)}\over \eps\delta^3}\le
 C\delta^{\kappa-1/2}+C\big[\delta+\eps^{1/2}\big]{\delta^{2\kappa-4}\over \eps} \end{equation*} and using \eqref{Lien2} yields
 \begin{equation}\label{EstGsPoutre}
\GG_s(u,B_{\eps,\delta}) \le C\eps^{\kappa^{'}}.\end{equation} Finally for any deformation $v\in \D_{\delta,\eps}$ and $u=v-I_d$ such that $J(v)\le 0$  we have
\begin{equation}\label{EstDist}
\begin{aligned}
&{\cal E}(u,{\cal S}_{\delta,\eps})\le C\delta^{2\kappa-1}=C q_\eps^2\eps^{2\kappa^{'}},\qquad  \hbox{and}\quad \int_{{\cal S}_{\delta,\eps}}f_{\delta}\cdot u\le C\delta^{2\kappa-1}.\end{aligned}\end{equation} Moreover, the above inequality together with  \eqref{coerW}show that \begin{equation}\label{EstW}
\int_{{\cal S}_{\delta,\eps}}\widehat{W}_\eps(\nabla  v)(x)dx \le C\delta^{2\kappa-1}\end{equation}
which in turn leads to
\begin{equation}\label{EstGSV}
 \big\|\nabla v^T\nabla v-\GI_3\big\|_{L^2(\Omega_{\delta} ; \R^{3\times 3})}\le  C\delta^{\kappa-1/2},\quad  \big\|\nabla v^T\nabla v-\GI_3\big\|_{L^2(B_{\eps,\delta} ; \R^{3\times 3})}\le  C\eps^{\kappa^{'}}\end{equation}
From   \eqref{EstDist} we also obtain
\begin{equation}\label{EstJ}
c\delta^{2\kappa-1}\le J_{\delta}(v)\le 0.
\end{equation}
We set
\begin{equation}\label{mdelta}
m_{\delta}=\inf_{v\in \D_{\delta,\eps}}J_{\delta}(v).
\end{equation}  In general, a minimizer of $J_{\delta}$ does not exist  on $\D_{\delta,\eps}$.
As a consequence of  \eqref{EstJ} we  have 
$$c\le {m_{\delta}\over \delta^{2\kappa-1}}\le 0.$$

\section{Limits of the Green-St Venant's strain tensors.}

In this subsection and the following one, we consider a sequence of deformations $(v_\delta)$ belonging to $\D_{\delta,\eps}$ and satisfying ($u_\delta=v_\delta-I_d$) 
\begin{equation}\label{energie}
{\cal E}(u_\delta,{\cal S}_{\delta,\eps})\le C\delta^{2\kappa-1}
\end{equation}
or equivalently
\begin{equation*}
{\cal E}(u_\delta,{\cal S}_{\delta,\eps})\le Cq_\eps^2\eps^{2\kappa^{'}}.
\end{equation*} Inequality \eqref{energie} implies
$$\GG_s(u_\delta , \Omega_\delta)\le C\delta^{\kappa-1/2}, \qquad \GG_s(u_\delta , B_{\eps,\delta}) \le C\eps^{\kappa^{'}}.$$
For any open subset ${\cal O}\subset\R^2$ and for any field $\psi\in H^1({\cal O};\R^3)$, we denote
\begin{equation}\label{NotGam}
\gamma_{\alpha\beta}(\psi)={1\over 2}\Big({\partial\psi_\alpha\over \partial x_\beta}+{\partial\psi_\beta\over \partial x_\alpha}\Big), \qquad (\alpha,\beta)\in\{1,2\}.
\end{equation}

\subsection{The rescaling operators.}

Before rescaling the domains, we introduce the reference domain $\Omega$ for the plate and the one $B$ for the rod
$$\Omega=\omega\times ]-1,1[,\qquad B=D\times ]0, L[=D(O, 1)\times ]0, L[.$$
As usual when dealing with thin  structures, we rescale $\Omega_\delta$ and $B_{\eps,\delta}$ using -for the plate-  the operator
$$\Pi_\delta( w)(x_1,x_2,X_3)=w(x_1,x_2,\delta X_3)\hbox{ for any}\;\; (x_1, x_2, X_3)\in \Omega $$
defined for e.g. $w\in L^2(\Omega_\delta)$ for which $\Pi_\delta (w)\in  L^2(\Omega)$ and using -for the rod- the operator
$$P_\eps( w)(X_1,X_2,x_3)=w(\eps X_1, \eps X_2, x_3)\hbox{ for any}\;\; (X_1, X_2, x_3)\in B $$
defined for e.g. $w\in L^2(B_{\varepsilon,\delta})$ for which $P_\eps (w)\in  L^2(B)$.

\subsection{Asymptotic behavior in the plate.}\label{In the plate.}

Following Section 2 we decompose the restriction of  $u_\delta=v_\delta-I_d$ to the plate.  The Theorem \ref{Theorem 3.3.} gives ${\cal U}_\delta$, ${\cal R}_\delta$ and $\overline{u}_\delta$, then estimates \eqref{Estm}  lead to the following convergences for  a subsequence still indexed by $\delta$ 
\begin{equation}\label{4.9}\begin{aligned}
{1\over \delta^{\kappa-2}} {\cal U}_{3,\delta} &\longrightarrow   {\cal U}_3 \quad \hbox{strongly in}\quad H^1(\omega),\\
{1\over \delta^{\kappa-1}}  {\cal U}_{\alpha,\delta} &\rightharpoonup    {\cal U}_\alpha \quad \hbox{weakly in}\quad H^1(\omega),\\
{1\over \delta^{\kappa-2}} {\cal R}_{\delta} &\rightharpoonup  {\cal R} \quad \hbox{weakly in}\quad H^1(\omega;\R^3),\\
{1\over \delta^{\kappa}}\Pi _\delta(\overline{u}_\delta)& \rightharpoonup  \overline{u} \quad \hbox{weakly in}\quad L^2(\omega;H^1(-1,1;\R^3),\\
{1\over \delta^{\kappa-1}}\Bigl({\partial{\cal U}_\delta\over \partial x_\alpha}-{\cal R}_\delta\land \Ge_\alpha\Big)& \rightharpoonup {\cal Z}_\alpha\quad \hbox{weakly in}\quad L^2(\omega;\R^3),\end{aligned}\end{equation}
The boundary conditions \eqref{CLUR} give here
\begin{equation}\label{CLURLimit}{\cal U}_3=0, \quad {\cal U}_\alpha=0,\quad  {\cal R}=0\qquad\hbox{on}\quad \gamma_0,
\end{equation}
while \eqref{4.9} show that ${\cal U}_3\in H^2(\omega)$ with 
\begin{equation}\label{4.100}
{\partial{\cal U}_3\over \partial x_1}=-{\cal R}_2, \qquad{\partial{\cal U}_3\over \partial x_2}={\cal R}_1.
\end{equation}
We also have
\begin{equation}\label{4.10}\begin{aligned}
{1\over \delta^{\kappa-1}}\Pi_\delta(u_{\alpha,\delta})&\rightharpoonup{\cal U}_\alpha- X_3{\partial {\cal U }_3 \over \partial x_\alpha} \quad \hbox{weakly in}\quad H^1(\Omega),\\
{1\over \delta^{\kappa-2}}\Pi_\delta(u_{3,\delta})&\longrightarrow{\cal U}_3 \quad \hbox{strongly in}\quad H^1(\Omega)
\end{aligned}\end{equation}
which shows that the rescaled limit displacement is a Kirchhoff-Love displacement.

 In \cite{BGJE}  the  limit of the Green-St Venant's strain tensor of the sequence $v_\delta$ is also derived. Let us set
\begin{equation}\label{4.101}
\overline{u}_p=\overline{u}+{X_3\over 2}\big({\cal Z}_1 \cdot\Ge_3\big)\Ge_1+{X_3\over 2}\big({\cal Z}_2 \cdot\Ge_3\big)\Ge_2
\end{equation}  and 
\begin{equation}\label{401}
{\cal Z}_{\alpha\beta}=\left\{
\begin{aligned}
&\gamma_{\alpha\beta}({\cal U})+{1\over 2}{\partial {\cal U}_3\over \partial x_\alpha}{\partial {\cal U}_3\over \partial x_\beta},\quad \hbox{if } \kappa=3,\\
&\gamma_{\alpha\beta}({\cal U})\hskip 3cm \hbox{if } \kappa>3.
\end{aligned}\right.
\end{equation} 
Then  we have
$${1\over 2\delta^{\kappa-1}}\Pi _\delta \big((\nabla_xv_\delta)^T\nabla_x v_\delta -\GI_3\big)\rightharpoonup
\GE_p\qquad\hbox{weakly in}\quad L^1(\Omega;\R^9),$$ where the symmetric matrix $\GE_p$ is defined by
\begin{equation}
\GE_p=\begin{pmatrix}
\displaystyle  -X_3{\partial^2{\cal U}_3\over \partial x_1^2}+{\cal Z}_{11} & \displaystyle  -X_3 {\partial^2 {\cal U}_3\over \partial x_1\partial x_2}+{\cal Z}_{12}
&\displaystyle  {1\over 2}{\partial\overline{u}_{p,1} \over \partial X_3}\\
* & \displaystyle  -X_3{\partial^2 {\cal U}_3\over \partial x_2^2}+{\cal Z}_{22}  &\displaystyle  {1\over 2}{\partial\overline{u}_{p,2} \over \partial X_3}\\
* & *&  \displaystyle  {\partial\overline{u}_{p,3} \over \partial X_3}
\end{pmatrix}\end{equation}

\subsection{Asymptotic behavior in the rod.}\label{In the rod.}

Now, we decompose the restriction of  $u_\delta=v_\delta-I_d$ to the rod.  The Theorem \ref{Theorem II.2.2.} gives ${\cal W}_\delta$, ${\cal Q}_\delta$ and $\overline{w}_\delta$, then the estimates  in \eqref{EstmR}, \eqref{Q(0)-W(0)} allow to claim that 
\begin{equation}\label{EstmRod2}
\begin{aligned}
&||\overline{w}_\delta ||_{L^2(B_{\varepsilon,\delta}; \R^3)}\le C \eps^{\kappa^{'}+1},\qquad||\nabla \overline{w}_\delta ||_{ L^2(B_{\varepsilon,\delta}; \R^3)}\le C \eps^{\kappa^{'}},\\
&||{\cal Q}_\delta -{\cal Q}_\delta (0)||_{H^1(-\delta,L;\R^3)}\le C\eps^{\kappa^{'}-2},\quad  \Bigl\|{d{\cal W}_\delta\over dx_3}-{\cal Q}_\delta \land\Ge_3\Big\|_{L^2(-\delta,L;\R^3)}\le C \eps^{\kappa^{'}-1}\\
&||{\cal W}_{\delta, 3}-{\cal W}_{\delta, 3}(0)||_{H^1(-\delta,L)}\le C\eps^{\kappa^{'}-1},\\
&||{\cal W}_\delta-{\cal W}_\delta(0)-{\cal Q}_\delta(0)x_3\land\Ge_3||_{H^1(-\delta,L;\R^3)}\le  C\eps^{\kappa^{'}-2}.
\end{aligned}
\end{equation}
 Moreover from \eqref{Q(0)-W(0)}  we get 
 \begin{equation}\label{QW1}\begin{aligned}
&|{\cal W}_{\alpha,\delta}(0)|\le  C\sqrt{\delta(\delta^2+ \eps^2)}\eps^{\kappa^{'}-2},\\
&|{\cal W}_{3,\delta}(0)-\widetilde {\cal U}_{3,\delta}(0,0)|\le C\sqrt{\delta+ \eps^2}\eps^{\kappa^{'}-1},\\
&||{\cal Q}_\delta(0)||_2\le C \sqrt{\delta+\eps}\eps^{\kappa^{'}-2}.
\end{aligned}\end{equation}
 \noindent Due to the above estimates we  are in a position to prove the  following lemma:
 
\begin{lemma}\label{Lemma 6.2. } There exists a subsequence still indexed by $\delta$ such that 
\begin{equation}\label{ConvPl}
\begin{aligned}
&{1\over \eps^{\kappa^{'}-2}}{\cal W}_{\alpha,\delta}\longrightarrow   {\cal W}_\alpha \quad \hbox{strongly in}\quad  H^1(0,L),\\
&{1\over \eps^{\kappa^{'}-1}}{\cal W}_{3,\delta}\rightharpoonup  {\cal W}_3 \quad \hbox{weakly in}\quad  H^1(0,L),\\
&{1\over \eps^{\kappa^{'}-2}}{\cal Q}_{\delta}\rightharpoonup{\cal Q} \quad \hbox{weakly in}\quad H^1(0,L ; \R^3), \\
&{1\over \eps^{\kappa^{'}}}P_\eps(\overline{ {w}_{\delta}}) \rightharpoonup  \overline{w} \quad \hbox{weakly in}\quad  L^2(0,L;H^1(D;\R^3)),\\
&{1\over \eps^{\kappa^{'}-1}}\Bigl({\partial{\cal W}_{\delta,1}\over \partial x_3}-{\cal Q}_{\delta,2} \Big)\rightharpoonup {\cal Z}_1\quad \hbox{weakly in}\quad  L^2(B),\\
&{1\over \eps^{\kappa^{'}-1}}\Bigl({\partial{\cal W}_{\delta,2}\over \partial x_3}+{\cal Q}_{\delta,1} \Big)\rightharpoonup {\cal Z}_2\quad \hbox{weakly in}\quad  L^2(B).
\end{aligned}
\end{equation} We also have ${\cal W}_\alpha\in H^2(0, L)$ and 
\begin{equation}\label{VIs}
{d{\cal W}_1\over dx_3}={\cal Q}_2,\qquad {d{\cal W}_2\over dx_3}=-{\cal Q}_1.
\end{equation}
The junction conditions
\begin{equation}\label{V=WQ(0)}
 {\cal W}_\alpha(0)=0,\quad {\cal Q}(0)=0,\qquad {\cal W}_3(0)={\cal U}_3(0,0)
\end{equation}
hold true.  Setting 
\begin{equation}\label{ubarcase2}
\overline{w}_r=\overline{w}+\big[X_1{\cal Z}_1+X_2 {\cal Z}_2\big]\Ge_3
\end{equation} we have 
\begin{equation}\label{GSVcase2}
{1\over 2\eps^{\kappa^{'}-1}}P_\eps \big((\nabla v_\delta)^T\nabla v_\delta-\GI_3\big)\rightharpoonup   \GE_r+\GF\qquad\hbox{weakly in}\quad  L^1(B ; \R^{3\times 3}),
\end{equation}  where the symmetric matrices $\GE_r$ and $\GF$ are defined by  
\begin{equation}\label{HatEcase2}
\begin{aligned}
\GE_r&= \begin{pmatrix}
\ds   \gamma_{11}(\overline{w}_r) & \ds   \gamma_{12}(\overline{w}_r) & \ds -{1\over 2}X_2{d{\cal Q}_3\over dx_3}+{1\over 2}{\partial\overline{w}_{r,3}\over \partial X_1}\\  \\
* & \ds   \gamma_{22}(\overline{w}_r) & \ds {1\over 2}X_1{d{\cal Q}_3\over dx_3}+{1\over 2}{\partial\overline{w}_{r,3}\over \partial X_2}\\  \\
* & * &  \ds -X_1{d^2{\cal U}_1\over dx^2_3}-X_2{d^2{\cal U}_2\over dx^2_3}+{d{\cal U}_3\over dx_3} &
\end{pmatrix},\\
\GF&=\left\{
\begin{aligned}
&{1\over 2}\big(||{\cal Q}||^2_2\GI_3-{\cal Q}.{\cal Q}^T\big)\quad \hbox{if} \; \kappa^{'}=3,\\
&0\hskip 3.5cm \hbox{ if } \; \kappa^{'}>3.
\end{aligned}\right.
\end{aligned}
 \end{equation}
\end{lemma}
\begin{proof}   
First, the estimates \eqref{EstmRod2} and \eqref{QW1} imply that the sequences $\ds{1\over \eps^{\kappa^{'}-2}}{\cal W}_{\alpha,\delta}$, $\ds{1\over \eps^{\kappa^{'}-1}}{\cal W}_{3,\delta}$,
$\ds{1\over \eps^{\kappa^{'}-2}}{\cal Q}_{\delta}$ are bounded in $ H^1(0,L ; \R^k)$, for $k=1$ or $k=3$. Taking into account also  \eqref{EstmRod2} and upon extracting a subsequence it follows that the convergences  \eqref{ConvPl} hold true together with \eqref{VIs}. The first  strong convergence in  \eqref{ConvPl} is in particular a consequence of \eqref{EstmRod2}.  The junction conditions on ${\cal Q}$ and  ${\cal W}_\alpha$ are immediate consequences of \eqref{QW1} and the convergences   \eqref{ConvPl}. 

In order to obtain the junction condition between the  bending in the plate and the stretching in the rod, note first that the  sequence $\ds{1\over \delta^{\kappa-2}}\widetilde{\cal U}_{\delta,3}$  converges strongly in $H^1(\omega)$ to ${\cal U}_3$ because of \eqref{V-Vtilde} and the first convergence in \eqref{4.9}. Besides this sequence is uniformly bounded in $H^2(D(O,\rho_0))$, hence it converges strongly to the same limit  ${\cal U}_3$ in $C^0(D(O,\rho_0))$.  
Moreover  the  weak convergence of the sequence $\ds{1\over \eps^{\kappa^{'}-1}}{\cal W}_{\delta,3}$  in $H^1(0, L)$, implies the convergence of $\ds{1\over \eps^{\kappa^{'}-1}}{\cal W}_{\delta,3}(0)$to ${\cal W}_{3}(0)$. Using the third  estimate in \eqref{QW1} gives the last condition in \eqref{V=WQ(0)}.

 Once the convergences \eqref{ConvPl} are established, the limit of the rescaled Green-St Venant strain tensor of the sequence $v_\delta$ is analyzed  in \cite{BGRod} and it gives \eqref{HatEcase2}.
\end{proof}
The above Lemma and the decomposition \eqref{DecR} lead to 
\begin{equation*}
\begin{aligned}
&{1\over \eps^{\kappa^{'}-2}}P_\eps( u_{\alpha,\delta})\longrightarrow   {\cal W}_\alpha \quad \hbox{strongly in}\quad  H^1(B),\\
&{1\over \eps^{\kappa^{'}-1}}P_\eps(u_{1,\delta}-{\cal W}_{1,\delta})\rightharpoonup  -X_2{\cal Q}_3 \quad \hbox{weakly in}\quad  H^1(B),\\
&{1\over \eps^{\kappa^{'}-1}}P_\eps(u_{2,\delta}-{\cal W}_{2,\delta})\rightharpoonup  X_1{\cal Q}_3 \quad \hbox{weakly in}\quad  H^1(B),\\
&{1\over \eps^{\kappa^{'}-1}}P_\eps(u_{3,\delta})\rightharpoonup  {\cal W}_3-X_1{d{\cal W}_1\over dx_3}-X_2{d{\cal W}_2\over dx_3} \quad \hbox{weakly in}\quad  H^1(B),\\
\end{aligned}
\end{equation*} 
which show that the limit rescaled displacement is a Bernoulli-Navier displacement.\\
\section{Asymptotic behavior of the sequence $\ds{m_\delta\over\delta^{2\kappa-1}}$.}

The goal of this section is to establish  Theorem \ref{theo9.1}. Let us first introduce a few notations. We set 
\begin{equation}\label{deflim2}
\begin{aligned}
\P\R_3 =\Big\{&({\cal U} ,{\cal W},{\cal Q}_3)\in H^1(\omega;\R^3)\times H^1(0,L ; \R^3) \times H^1(0,L)\;|\; \\
&{\cal U}_3\in H^2(\omega),\quad {\cal W}_\alpha\in H^2(0,L),\quad {\cal U}=0,\quad {\partial {\cal U}_3\over \partial x_\alpha}=0 \quad  \hbox{on} \enskip  \gamma_0,\\
& {\cal W}_3(0)={\cal U}_3(0,0),\qquad  {\cal W}_\alpha(0)={d{\cal W}_\alpha\over dx_3}(0)={\cal Q}_3(0)=0\Big\}
\end{aligned}\end{equation} 
We introduce below the "limit" rescaled elastic energies for the plate and the rod 
\begin{equation}\label {JP}
\begin{aligned}
 {\cal J}_p({\cal U})&= {E_p\over 3(1-\nu_p^2)}\int_\omega\Big[(1-\nu_p)\sum_{\alpha,\beta=1}^2\Big|{\partial^2{\cal U}_3\over \partial x_\alpha\partial x_\beta}\Big|^2+\nu_p\big(\Delta{\cal U}_3\big)^2\Big]\\
&+{E_p\over (1-\nu_p^2)}\int_\omega\Big[(1-\nu_p)\sum_{\alpha,\beta=1}^2\big|{\cal Z}_{\alpha\beta}\big|^2+\nu_p\big({\cal Z}_{11}+{\cal Z}_{22}\big)^2\Big],\\
{\cal J}_r({\cal W},{\cal Q}_3) &= {E_r\pi\over 8}\int_0^L\Big[ \Big|{d^2{\cal W}_1 \over dx_3^2}\Big|^2+\Big|{d^2{\cal W}_2 \over dx_3^2}\Big|^2\Big]+{E_r\pi\over 2}\Big|{d{\cal W}_3 \over dx_3}+\GF_{33}\Big|^2\\
&+{\mu_r\pi\over 8}\int_0^L\Big|{d{\cal Q}_3 \over dx_3}\Big|^2
\end{aligned}\end{equation}
where  the ${\cal Z}_{\alpha\beta}$'s are given by
\begin{equation*}
{\cal Z}_{\alpha\beta}=\left\{
\begin{aligned}
&\gamma_{\alpha\beta}({\cal U})+{1\over 2}{\partial {\cal U}_3\over \partial x_\alpha}{\partial {\cal U}_3\over \partial x_\beta},\quad \hbox{if } \kappa=3,\\
&\gamma_{\alpha\beta}({\cal U})\hskip 3cm \hbox{if } \kappa>3.
\end{aligned}\right.
\end{equation*}  and where $\GF_{33}$ is given by 
\begin{equation}
\GF_{33}=\left\{\begin{aligned}
&{1\over 2}\Big(\Big|{d{\cal W}_1 \over dx_3}\Big|^2+\Big|{d{\cal W}_2 \over dx_3}\Big|^2\Big)\enskip \hbox{if}\enskip \kappa^{'}=3,\\
&0 \hskip 4.0cm \hbox{if}\enskip \kappa^{'}>3.
\end{aligned}\right.
\end{equation}
The total energy of the plate-rod structure is given by  the functional ${\cal J}$ defined over $\P\R_3$
\begin{equation}\label {J2}
 {\cal J}_3({\cal U},{\cal W},{\cal Q}_3)= {\cal J}_p({\cal U})+{\cal J}_r({\cal W},{\cal Q}_3) -{\cal L}_3({\cal U},{\cal W},{\cal Q}_3)
\end{equation} 
with
\begin{equation}\label{FormLin}
\begin{aligned}
 {\cal L}_3({\cal U},{\cal W},{\cal Q}_3)&=2\int_{\omega} f_p\cdot {\cal U} +\pi \int_0^L f_r\cdot {\cal W}dx_3+{\pi \over 2} \int_0^L{g_\alpha}\cdot\big({\cal Q}\land\Ge_\alpha\big)dx_3
 \end{aligned}
 \end{equation}  where
 \begin{equation}\label{SR}
  {\cal Q}=- {d{\cal W}_2\over dx_3}\Ge_1+ {d{\cal W}_1\over dx_3}\Ge_2+{\cal Q}_3\Ge_3. 
\end{equation}
It is worth noting that  the functional ${\cal J}_p({\cal U})$ corresponds to the elastic energy of a Von K\'arm\'an plate model for $\kappa=3$ (see e.g. \cite{Ciarlet3}) and to the classical linear plate model for $\kappa> 3$. Similarly the functional ${\cal J}_r({\cal W},{\cal Q}_3)$ corresponds to a nonlinear rod model derived in  \cite{BGRod} for $\kappa^{'}=3$ and to the classical linear rod model for  $\kappa^{'}>3$. Let us also notice that in the space $\P\R_3$ the bending in the plate is equal to the stretching in the rod at the junction while the bending and the section-rotation  of the rod in the junction are equal to $0$ (see \eqref{SR}). \\

In the lemma  below we give sufficient conditions on the applied forces in order to insure the existence of at least a minimizer of  $ {\cal J}$ (see \cite{Ciarlet3} for a proof of the result for different boundary conditions for the displacement on $\partial \omega$). 
\begin{lemma}\label{LemExistence}  We have

\noindent $\bullet$ if $\kappa>3$ and $\kappa^{'}>3$ then the minimization problem
\begin{equation}\label{M44}
\min_{({\cal U},{\cal W},{\cal Q}_3)\in\P\R_3} {\cal J}_3({\cal U},{\cal W},{\cal Q}_3)
\end{equation}
admits an unique solution,

\noindent $\bullet$ if $\kappa=3$ and $\kappa^{'}>3$ then there exists a constant $C^{*}_l$  such that, if $(f_{p1},f_{p2})$ satisfies

 \begin{equation}\label{CstPl}
||f_{p1}||^2_{L^2(\omega)}+||f_{p2}||^2_{L^2(\omega)} <  C^*_l
\end{equation} then \eqref{M44} admits at least a solution,

\noindent $\bullet$ if $\kappa>3$ and $\kappa^{'}=3$ then there exists a constant $C^{**}_l$  such that, if $f_{r3}$ satisfies

 \begin{equation}\label{CstPt}
||f_{r3}||_{L^2(0,L)} <  C^{**}_l
\end{equation} then \eqref{M44} admits at least a solution,

\noindent $\bullet$ if $\kappa=3$ and $\kappa^{'}=3$ then  if the applied forces $(f_{p1},f_{p2})$ and $f_{r3}$  satisfy \eqref{CstPl} and \eqref{CstPt}  then \eqref{M44} admits at least a solution.
\end{lemma}
\begin{proof} First, in the case  $\kappa>3$ and $\kappa^{'}>3$ the result is well known.

\noindent We prove the lemma in the case $\kappa=3$ and $\kappa^{'}=3$. The two other cases are simpler and left to the reader.

\noindent  Due to the boundary conditions on ${\cal U}_3$ in $\P\R_3$, we immediately have
\begin{equation}\label{NU3}
||{\cal U}_3||^2_{H^2(\omega)}\le C{\cal J}_p({\cal U}).
\end{equation} Then we get
\begin{equation}\label{NU4}
\begin{aligned}
\sum_{\alpha,\beta=1}^2||\gamma_{\alpha,\beta}({\cal U})||^2_{L^2(\omega)} & \le {\cal J}_p({\cal U})+ C\big\|\nabla{\cal U}_3\big\|^2_{L^4(\omega;\R^2)}\\
& \le {\cal J}_p({\cal U})+ C[{\cal J}_p({\cal U})]^2.
\end{aligned}\end{equation} Thanks to the 2D Korn's inequality we obtain
\begin{equation}\label{NU5}
||{\cal U}_1||^2_{H^1(\omega)}+||{\cal U}_2||^2_{H^1(\omega)}  \le C{\cal J}_p({\cal U})+ C_{P}[{\cal J}_p({\cal U})]^2.
\end{equation} 

\noindent  Again, due to the boundary conditions on ${\cal W}_\alpha$ and ${\cal Q}_3$ in $\P\R_3$, we immediately have
\begin{equation}\label{NU6}
||{\cal W}_1||^2_{H^2(0,L)}+||{\cal W}_2||^2_{H^2(0,L)}+||{\cal Q}_3||^2_{H^1(0,L)}\le {\cal J}_r({\cal W},{\cal Q}_3).
\end{equation} Then we get
\begin{equation}\label{NU7}
\begin{aligned}
\Big\|{d{\cal W}_3\over dx_3}\Big\|^2_{L^2(0,L)} & \le {\cal J}_r({\cal W},{\cal Q}_3)+ C\Bigl\{\Big\|{d{\cal W}_1\over dx_3}\Big\|^2_{L^4(0,L)}+\Big\|{d{\cal W}_2\over dx_3}\Big\|^2_{L^4(0,L)}\Big\}\\
& \le {\cal J}_r({\cal W},{\cal Q}_3)+ C[{\cal J}_r({\cal W},{\cal Q}_3)]^2.
\end{aligned}\end{equation} From the above inequality and\eqref{NU3} we obtain
\begin{equation}\label{NU8}
\begin{aligned}
\big\|{\cal W}_3\big\|^2_{L^2(0,L)} &\le C|{\cal W}_3(0)|^2+C\Big\|{d{\cal W}_3\over dx_3}\Big\|^2_{L^2(0,L)}\\
& \le C{\cal J}_p({\cal U})+C{\cal J}_r({\cal W},{\cal Q}_3)+  C_R[{\cal J}_r({\cal W},{\cal Q}_3)]^2.
\end{aligned}\end{equation}

Since $ {\cal J}_3(0, 0, 0)=0$, let us consider a minimizing sequence $({\cal U}^{(N)},{\cal W}^{(N)},{\cal Q}^{(N)}_3)\in \P\R_3$ satisfying ${\cal J}_3({\cal U}^{(N)},{\cal W}^{(N)},{\cal Q}^{(N)}_3)\le 0$ 
\begin{equation*}
m=\inf_{({\cal U},{\cal W},{\cal Q}_3)\in\P\R_3}{\cal J}_3({\cal U},{\cal W},{\cal Q}_3)=\lim_{N\to+\infty}{\cal J}_3({\cal U}^{(N)},{\cal W}^{(N)},{\cal Q}^{(N)}_3)
\end{equation*} where $m\in [-\infty, 0]$.

\noindent With the help of \eqref{NU3}-\eqref{NU8} we get
\begin{equation}
\begin{aligned}
&\hskip-5mm{\cal J}_p( {\cal U}^{(N)})+ {\cal J}_r({\cal W}^{(N)}, {\cal Q}^{(N)}_3)\le  C ||f_{p3}||\sqrt{{\cal J}_p({\cal U}^{(N)})}\\
&+\big(||f_{p1}||^2_{L^2(\omega)}+||f_{p2}||^2_{L^2(\omega)}\big)^{1/2}\big(C\sqrt{{\cal J}_p({\cal U}^{(N)})}+\sqrt{C_P}{\cal J}_p({\cal U}^{(N)})\big)\\
&+\sum_{\alpha=1}^2\big(||f_{r\alpha}||_{L^2(0,L)}+||g_{\alpha}||_{L^2(0,L;\R^3)}\big)\sqrt{{\cal J}_r({\cal W}^{(N)},{\cal Q}^{(N)}_3)}\\
&+||f_{r3}||_{L^2(0,L)}\big(C\sqrt{{\cal J}_r({\cal W}^{(N)},{\cal Q}^{(N)}_3)}+C\sqrt{{\cal J}_p({\cal U}^{(N)})}+\sqrt{C_R}{\cal J}_r({\cal W}^{(N)},{\cal Q}^{(N)}_3)\big)
\end{aligned}\end{equation} Choosing $\ds C^*_l={1\over C_P}$ and $\ds C^{**}_R={1\over \sqrt{C_R}}$, if the applied forces satisfy \eqref{CstPl} and \eqref{CstPt}  then the following estimates hold true
\begin{equation}
\begin{aligned}
&||{\cal U}^{(N)}_3||_{H^2(\omega)}+||{\cal U}^{(N)}_1||_{H^1(\omega)}+||{\cal U}^{(N)}_2||_{H^1(\omega)} +||{\cal W}^{(N)}_1||_{H^2(0,L)}\\
+ &||{\cal W}^{(N)}_2||_{H^2(0,L)}+||{\cal Q}^{(N)}_3||_{H^1(0,L)}+||{\cal W}^{(N)}_3||_{H^1(0,L)}\le C
\end{aligned}\end{equation} where the constant $C$ does not depend on $N$.
\vskip 1mm
As a consequence, there exists $({\cal U}^{(*)},{\cal W}^{(*)},{\cal Q}^{(*)}_3)\in \P\R_3$ such that for a subsequence 
\begin{equation*}
\begin{aligned}
{\cal U}^{(N)}_3 &\rightharpoonup {\cal U}^{(*)}_3\quad \hbox{weakly in}\enskip H^2(\omega) \;\hbox{and strongly in }\; W^{1,4}(\omega),\\
{\cal U}^{(N)}_\alpha &\rightharpoonup {\cal U}^{(*)}_\alpha\quad \hbox{weakly in}\enskip H^1(\omega),\\
{\cal W}^{(N)}_\alpha&\rightharpoonup {\cal W}^{(*)}_\alpha\quad \hbox{weakly in}\enskip H^2(0,L) \;\hbox{and strongly in }\; W^{1,4}(0,L),\\
{\cal Q}^{(N)}_3&\rightharpoonup {\cal Q}^{(*)}_3\quad \hbox{weakly in}\enskip H^1(0,L),\\
{\cal W}^{(N)}_3&\rightharpoonup {\cal W}^{(*)}_3\quad \hbox{weakly in}\enskip H^1(0,L).
\end{aligned}\end{equation*} Finally, since ${\cal J}_3$ is weakly sequentially continuous in 
$$H^2(\omega)\times H^1(\omega;\R^2) \times L^2(\Omega;\R^3)\times H^2(0,L;\R^2)\times H^1(0,L;\R^2)\times L^2(0,L)$$ with respect to 
$$({\cal U}_3, {\cal U}_1, {\cal U}_2, {\cal Z}_{11}, {\cal Z}_{12}, {\cal Z}_{22}, {\cal W}_1, {\cal W}_2, {\cal W}_3,{\cal Q}_3,  \GF_{33})$$ The above weak and strong converges imply that
\begin{equation*}
{\cal J}_3({\cal U}^{(*)},{\cal W}^{(*)},{\cal Q}^{(*)}_3)=m
=\min_{({\cal U},{\cal W},{\cal Q}_3)\in\P\R_3}{\cal J}_3({\cal U},{\cal W},{\cal Q}_3)
\end{equation*} which ends the proof of the lemma.
\end{proof}

The following theorem is the main result of the paper. It characterizes the limit of the rescaled infimum of the total energy $\ds {m_\delta\over \delta^{2\kappa-1}}={1\over \delta^{2\kappa-1}}\inf_{v\in \D_{\delta,\eps}}J_{\delta}(v)$  as the minimum of  the limit energy ${\cal J}_3$ over the space $\P\R_3$. Due to the  conditions on the fields ${\cal U},{\cal W},{\cal Q}_3$  in $\P\R_3$, this minimization problem modelizes the junction of a 2d plate model with a 1d rod model of the type "plate bending-rod stretching".

\begin {theorem}\label{theo9.1} Under the assumptions \eqref{ForceP},  \eqref{Cstar}- \eqref{Cstarstar}and \eqref{CstPl}-\eqref{CstPt} on the forces, we have 
\begin{equation}\label{res2}
\lim_{\delta\to 0}{m_{\delta}\over \delta^{2\kappa-1}}=\min_{({\cal U},{\cal W},{\cal Q}_3)\in \P\R_3 } {\cal J}_3({\cal U},{\cal W},{\cal Q}_3),
\end{equation}
where the functional ${\cal J}$ is defined by \eqref{J2}.
\end{theorem}
\begin{proof}
{\it Step 1.} In this step we show that 
\begin{equation}\label{step1}
\min_{({\cal U},{\cal W},{\cal Q}_3)\in \P\R_3 }  {\cal J}_3({\cal U},{\cal W},{\cal Q}_3)\le \liminf_{\delta\to0}{m_\delta\over \delta^{2\kappa-1}}.
\end{equation}

Let   $(v_\delta)_\delta$  be  a sequence of deformations belonging to $\D_{\delta,\eps}$ and such that
\begin{equation}\label{Hypvdelta2}
\lim_{\delta\to 0}{J_{\delta}(v_\delta)\over \delta^{2\kappa-1}}=\liminf_{\delta\to0}{m_\delta\over \delta^{2\kappa-1}}.
\end{equation} One can always assume that $J_{\delta}(v_\delta)\le 0$ without loss of generality. From the analysis of the previous section and, in particular from  estimates \eqref{EstDist} the sequence $v_\delta$ satisfies
 \begin{equation}
 \begin{aligned}\label{6.54}
&{\cal E}(u_\delta,{\cal S}_{\delta,\eps})\le C\delta^{2\kappa-1}=C q_\eps^2\eps^{2\kappa^{'}},\qquad  ||\hbox{dist}(\nabla  v_\delta , SO(3))||_{L^2(\Omega_{\delta})}\le C\delta^{\kappa-1/2}\\
&||\hbox{dist}(\nabla  v_\delta , SO(3))||_{L^2(B_{\eps,\delta})}\le C\eps^{\kappa^{'}}.\end{aligned}\end{equation} 
Estimates \eqref{EstGSV} give
 \begin{equation}\label{6.55}
 \big\|\nabla v^T_\delta\nabla v_\delta-\GI_3\big\|_{L^2(\Omega_{\delta} ; \R^{3\times 3})}\le  C\delta^{\kappa-1/2},\quad  \big\|\nabla v^T_\delta\nabla v_\delta-\GI_3\big\|_{L^2(B_{\eps,\delta} ; \R^{3\times 3})}\le  C\eps^{\kappa^{'}}.
 \end{equation} 

 Firstly, for any fixed $\delta$, the displacement  $u_\delta=v_\delta-I_d$, restricted to $\Omega_\delta$, is decomposed as in Theorem \ref{Theorem 3.3.}.  Due to the second estimate in   \eqref{6.54}, we can apply the results of Subsection \ref{In the plate.} to the sequence $(v_\delta)$. As a consequence  there exist a subsequence (still indexed by $\delta$) and ${\cal U}^{(0)},\;{\cal R}^{(0)}\in H^1(\omega;\R^3)$,  such that the convergences \eqref{4.9} and \eqref{4.10} hold true. Due to \eqref{CLURLimit} and \eqref{4.100} the field ${\cal U}_3$ belongs to $H^2(\omega)$, and we have the boundary conditions 
\begin{equation}\label{CLU0}{\cal U}^{(0)}=0,\quad \nabla U^{(0)}_3=0,\qquad\hbox{on}\quad \gamma_0,
\end{equation}

Subsection  \ref{In the plate.} also  shows that there exits $\overline{u}_p^{(0)}\in L^2(\omega;H^1(-1,1;\R^3))$ such that 
 \begin{equation}\label{E0}
 {1\over 2\delta^{\kappa-1}}\big(\nabla v_\delta^T\nabla v_\delta-\GI_3\big)\rightharpoonup \GE^{(0)}_p\quad \hbox{weakly in }\enskip L^1(\Omega;\R^9)
 \end{equation}
 where $\GE^{(0)}_p$ is defined 
 \begin{equation}\label{E00}
\GE^{(0)}_p=\begin{pmatrix}
\displaystyle  -X_3{\partial^2{\cal U}^{(0)}_3\over \partial x_1^2}+{\cal Z}^{(0)}_{11} & \displaystyle  -X_3 {\partial^2 {\cal U}^{(0)}_3\over \partial x_1\partial x_2}+{\cal Z}^{(0)}_{12}
&\displaystyle  {1\over 2}{\partial\overline{u}^{(0)}_{p,1} \over \partial X_3}\\
* & \displaystyle  -X_3{\partial^2 {\cal U}^{(0)}_3\over \partial x_2^2}+{\cal Z}^{(0)}_{22}  &\displaystyle  {1\over 2}{\partial\overline{u}^{(0)}_{p,2} \over \partial X_3}\\
* & *&  \displaystyle  {\partial\overline{u}^{(0)}_{p,3} \over \partial X_3}
\end{pmatrix}\end{equation}
with  
\begin{equation}\label{Z0}
{\cal Z}^{(0)}_{\alpha\beta}=\left\{
\begin{aligned}
&\gamma_{\alpha\beta}({\cal U}^{(0)})+{1\over 2}{\partial {\cal U}^{(0)}_3\over \partial x_\alpha}{\partial {\cal U}^{(0)}_3\over \partial x_\beta},\quad \hbox{if } \kappa=3,\\
&\gamma_{\alpha\beta}({\cal U}^{(0)})\hskip 3cm \hbox{if } \kappa>3.
\end{aligned}\right.
\end{equation} Moreover thanks to the first estimate in \eqref{6.55}, the weak convergence \eqref{E0} actually occurs in $L^2(\Omega;\R^9)$.

Secondly, still for $\delta$  fixed,  the displacement  $u_\delta=v_\delta-I_d$, restricted to $B_{\eps,\delta}$, is decomposed as in Theorem \ref{Theorem 3.3.}.  Again due  to the third estimate in   \eqref{6.55}, we can apply the results of Subsection \ref{In the rod.} to the sequence $(v_\delta)$. As a consequence  there exist a subsequence (still indexed by $\delta$) and ${\cal W}^{(0)},\;{\cal Q}^{(0)}_3\in H^1(0,L;\R^3)$,  such that the convergences \eqref{ConvPl}. As a consequence of  \eqref{VIs} the fields ${\cal W}^{(0)}$ belongs to $H^2(0,L)$ and we have
$${d{\cal W}^{(0)}\over dx_3}={\cal Q}^{(0)}_3\land\Ge_3.$$
The junction conditions \eqref{V=WQ(0)} and \eqref{V=WQ(0)} give
\begin{equation}\label{CJUQ0}
{\cal Q}^{(0)}(0)=0,\quad {\cal W}^{(0)}_\alpha(0)=0,\qquad {\cal W}^{(0)}_3(0)={\cal U}^{(0)}_3(0,0).\end{equation}
\vskip 1mm
\noindent The triplet $({\cal U}^{(0)}, {\cal W}^{(0)},{\cal Q}^{(0)}_3)$ belongs to $\P\R_3$.

 Subsection  \ref{In the rod.} also  shows that there exits $\overline{w}_r^{(0)}\in L^2(0,L;H^1(D;\R^3))$ such that 
 \begin{equation}\label{612}
{1\over 2\eps^{\kappa^{'}-1}}P_\eps \big((\nabla v_\delta)^T\nabla v_\delta-\GI_3\big)\rightharpoonup   \GE^{(0)}_r\qquad\hbox{weakly in}\quad  L^1(B ; \R^{3\times 3}),
\end{equation}  where the symmetric matrices $\GE^{(0)}_r$ is defined by  
\begin{equation}\label{613}
\GE^{(0)}_r= \begin{pmatrix}
\ds   \gamma_{11}(\overline{w}_r^{(0)}) & \ds   \gamma_{12}(\overline{w}_r^{(0)}) & \ds -{1\over 2}X_2{d{\cal Q}^{(0)}_3\over dx_3}+{1\over 2}{\partial\overline{w}^{(0)}_{r,3}\over \partial X_1}\\  \\
* & \ds   \gamma_{22}(\overline{w}_r^{(0)}) & \ds {1\over 2}X_1{d{\cal Q}^{(0)}_3\over dx_3}+{1\over 2}{\partial\overline{w}^{(0)}_{r,3}\over \partial X_2}\\  \\
* & * &  \ds -X_1{d^2{\cal U}^{(0)}_1\over dx^2_3}-X_2{d^2{\cal U}^{(0)}_2\over dx^2_3}+{d{\cal U}^{(0)}_3\over dx_3} &
\end{pmatrix}+\GF^{(0)},
\end{equation}
\begin{equation}\label{F0}
\begin{aligned}
\GF^{(0)}&=\left\{
\begin{aligned}
&{1\over 2}\big(||{\cal Q}^{(0)}||^2_2\GI_3-{\cal Q}^{(0)}\big({\cal Q}^{(0)}\big)^T\big)\quad \hbox{if} \; \kappa^{'}=3,\\
&0\hskip 5cm \hbox{ if } \; \kappa^{'}>3,
\end{aligned}\right.\\
\hbox{where}\;\; {\cal Q}^{(0)}&=\ds- {d{\cal W}^{(0)}_2\over dx_3}\Ge_1+{d{\cal W}^{(0)}_1\over dx_3}\Ge_2+{\cal Q}^{(0)}_3\Ge_3.
\end{aligned} \end{equation} 
Moreover thanks to the second estimate in \eqref{6.55}, the weak convergence \eqref{612} actually occurs in $L^2(B;\R^9)$. 

\noindent First of all, we have
\begin{equation*}
\begin{aligned}
{1\over \delta^{2\kappa-1}} \int_{{\cal S}_{\delta,\eps}}\widehat{W}_\eps\big(\nabla v_\delta\big)&= {1\over \delta^{2\kappa-1}} \int_{{\Omega}_{\delta}}\widehat{W}_\eps\big(\nabla v_\delta\big)+{1\over q^2_\eps\eps^{2\kappa^{'}}} \int_{B_{\eps,\delta}\setminus C_{\delta,\eps}}\widehat{W}_\eps\big(\nabla v_\delta\big)\\
= \int_{\Omega}Q_p\Big(\Pi_\delta\Big[{1\over \delta^{\kappa-1}}&\big((\nabla v_\delta)^T\nabla v_\delta-\GI_3\Big]\Big)+ \int_{B} Q_r\Big( \chi_{B\setminus D\times ]0,\delta[}P_\eps\Big[{1\over \eps^{\kappa^{'}-1}}\big((\nabla v_\delta)^T\nabla v_\delta-\GI_3\Big]\Big)\\
\end{aligned}
\end{equation*} From the weak convergences  of the Green-St Venant's tensors in \eqref{E0} and \eqref{612} (recall that these convergences hold true in $L^2$) and the limit of the term involving the forces \eqref{limitforces} we obtain 
\begin{equation}\label{6.100}
\liminf_{\delta \to 0}{J_\delta(v_\delta)\over \delta^{2\kappa-1}}\ge\int_{\Omega}Q\big(\GE^{(0)}_p\big)+ \int_{B} Q\big(\GE^{(0)}_r\big)- \lim_{\delta\to 0}{1\over \delta^{2\kappa-1}}\int_{{\cal S}_{\delta,\eps}} f_{\delta}\cdot(v_\delta -I_d)
\end{equation} where $\GE^{(0)}_p$ and $\GE^{(0)}_r$ are given by \eqref{E00} and \eqref{613}.
In order to derive the last limit in \eqref{6.100} we use the assumptions on the forces \eqref{ForceP} and the convergences \eqref{4.9} and \eqref{ConvPl} and this leads to 
\begin{equation}\label{limitforces}
 \lim_{\delta\to 0}{1\over \delta^{2\kappa-1}}\int_{{\cal S}_{\delta,\eps}} f_{\delta}\cdot(v_\delta -I_d)= {\cal L}_3({\cal U}^{(0)},{\cal W}^{(0)},{\cal Q}^{(0)}_3)
 \end{equation}   where
$ {\cal L}_3({\cal U},{\cal W},{\cal Q}_3)$ is given by \eqref{FormLin} for any triplet in $\P\R_3$. 
From \eqref{6.100} and \eqref{limitforces}, we obtain 
\begin{equation}\label{6.101}
\liminf_{\delta \to 0}{J_\delta(v_\delta)\over \delta^{2\kappa-1}}\ge\int_{\Omega}Q\big(\GE^{(0)}_p\big)+ \int_{B} Q\big(\GE^{(0)}_r\big)-{\cal L}_3({\cal U}^{(0)},{\cal W}^{(0)},{\cal Q}^{(0)}_3).
\end{equation}
The next step in the derivation of the limit energy consists in minimizing $\int_{-1}^1Q_p\big(\GE^{(0)}_p\big)dX_3$ (resp. $\int_D Q_r\big(\GE^{(0)}_r\big)dX_1dX_2$) with respect to $\overline{u}^{(0)}_p$( resp. $\overline{w}_r^{(0)}$).

First the expressions of $Q_p$ and of $\GE^{(0)}_p$ under a few calculations show that 
\begin{equation}\label{605}
\begin{aligned}
\int_{-1}^1Q_p\big(\GE^{(0)}_p\big)dX_3\ge &{E_p\over 3(1-\nu_p^2)}\Big[(1-\nu_p)\sum_{\alpha,\beta=1}^2\Big|{\partial^2{\cal U}^{(0)}_3\over \partial x_\alpha\partial x_\beta}\Big|^2+\nu_p\big(\Delta{\cal U}^{(0)}_3\big)^2\Big]\\
&+{E_p\over (1-\nu_p^2)}\Big[(1-\nu_p)\sum_{\alpha,\beta=1}^2\big|{\cal Z}^{(0)}_{\alpha\beta}\big|^2+\nu_p\big({\cal Z}^{(0)}_{11}+{\cal Z}^{(0)}_{22}\big)^2\Big]
\end{aligned}\end{equation}
the expression in the right hand side of \eqref{605} is obtained through replacing $\overline{u}^{(0)}_p$ by 
\begin{equation}\label{barup}
\overline{\overline{u}}^{(0)}_p(\cdot ,\cdot ,X_3)={\nu_p\over 1-\nu_p}\Big[\Big({X_3^2\over 2}-{1\over 6}\Big)\Delta {\cal U}^{(0)}_3-X_3\big({\cal Z}^{(0)}_{11}+{\cal Z}^{(0)}_{22}\big)\Big]\Ge_3.
\end{equation}
Then the expressions of  $Q_r$ and of $\GE^{(0)}_r$ permit to obtain
\begin{equation}\label{606}
\begin{aligned}
\int_D Q_r\big(\GE^{(0)}_r\big)dX_1dX_2\ge &{E_r\pi\over 8}\Big[ \Big|{d^2{\cal W}^{(0)}_1 \over dx_3^2}\Big|^2+\Big|{d^2{\cal W}^{(0)}_2 \over dx_3^2}\Big|^2\Big]+{E_r\pi\over 2}\Big|{d{\cal W}^{(0)}_3 \over dx_3}+\GF^{(0)}_{33}\Big|^2\\
&+{\mu_r\pi\over 8}\Big|{d{\cal Q}^{(0)}_3 \over dx_3}\Big|^2
\end{aligned}\end{equation}
and similarly  the expression in the right hand side of \eqref{606} is derived through replacing $\overline{w}_r^{(0)}$ by 
\begin{equation}\label{barur}\begin{aligned}
\overline{\overline{w}}^{(0)}_{r,1}&=-\nu_r\Big[{X_2^2-X_1^2 \over 2}{d^2{\cal W}^{(0)}_1 \over dx_3^2}-X_1X_2{d^2{\cal W}^{(0)}_2 \over dx_3^2}+X_1\Big({d{\cal W}^{(0)}_3 \over dx_3}+\GF^{(0)}_{33}\Big)\Big]-X_1\GF^{(0)}_{11}-{X_2\over 2}\GF^{(0)}_{12}\\
\overline{\overline{w}}^{(0)}_{r,2}&=-\nu_r\Big[{X_1^2-X_2^2 \over 2}{d^2{\cal W}^{(0)}_2 \over dx_3^2}-X_1X_2{d^2{\cal W}^{(0)}_1 \over dx_3^2}+X_2\Big({d{\cal W}^{(0)}_3 \over dx_3}+\GF^{(0)}_{33}\Big)\Big]-{X_1\over 2}\GF^{(0)}_{12}-X_2\GF^{(0)}_{22}\\
\overline{\overline{w}}^{(0)}_{r,3}&=-X_1\GF^{(0)}_{13} -X_2\GF^{(0)}_{23}.
\end{aligned}\end{equation}
In view of \eqref{6.101}, \eqref{605} and \eqref{606}, the proof of\eqref{step1} is achieved.
\vskip 2mm
\noindent {\it Step  2.} Under the assumptions \eqref{CstPl}-\eqref{CstPt}, we know that there exists  $({\cal U}^{(1)},{\cal W}^{(1)},{\cal Q}^{(1)}_3)\in \P\R_3$ such that 
$$ \min_{({\cal U},{\cal W},{\cal Q}_3)\in \P\R_3}{\cal J}_3({\cal U},{\cal W},{\cal Q}_3)={\cal J}_3({\cal U}^{(1)},{\cal W}^{(1)},{\cal Q}^{(1)}_3).$$ Now, in this step we show that 
$$ \limsup_{\delta\to0}{m_\delta\over \delta^{2\kappa-1}}\le {\cal J}_3({\cal U}^{(1)},{\cal W}^{(1)},{\cal Q}^{(1)}_3).$$

Let $\overline{\overline{u}}^{(1)}_p$ be in $L^2(\omega;H^1(-1,1;\R^3))$ obtained through replacing ${\cal U}^{(0)}$ by ${\cal U}^{(1)}$ in \eqref{Z0}-\eqref{barup}  and $\overline{\overline{w}}^{(1)}_r$ be in $L^2(0,L;H^1(D;\R^3))$ obtained through replacing ${\cal W}^{(0)}$ and ${\cal Q}^{(0)}_3$ by ${\cal U}^{(1)}$ and ${\cal Q}^{(0)}_3$ in \eqref{F0}- \eqref{barur}.

\noindent We now consider a sequence $\big({\cal U}^{(n)},{\cal W}^{(n)},{\cal Q}^{(n)}_3, \overline{u}^{(n)}, \overline{w}^{(n)}\big)_{n\ge 2}$ such that

$\bullet$  ${\cal U}^{(n)}_\alpha \in W^{2,\infty}(\omega)\cap H^1_{\gamma_0}(\omega)$ and 
$${\cal U}^{(n)}_\alpha \longrightarrow {\cal U}^{(1)}_\alpha \hbox{ strongly in } H^1(\omega),$$

$\bullet$ ${\cal U}^{(n)}_3 \in W^{3,\infty}(\omega)\cap H^2_{\gamma_0}(\omega)$ and 
$${\cal U}^{(n)}_3 \longrightarrow {\cal U}^{(1)}_3 \hbox{ strongly in } H^2(\omega),$$

$\bullet$  ${\cal W}^{(n)}_\alpha \in W^{3,\infty}(-1/n,L)$ with ${\cal W}^{(n)}_\alpha =0$ in $[-1/n,1/n]$  and
$${\cal W}^{(n)}_\alpha \longrightarrow {\cal W}^{(1)}_\alpha \hbox{ strongly in } H^2(0,L),$$

$\bullet$  ${\cal W}^{(n)}_3 \in W^{2,\infty}(-1/n,L)$ with ${\cal W}^{(n)}_3 ={\cal U}^{(n)}_3(0,0)$ in $[-1/n,1/n]$  and
$${\cal W}^{(n)}_3 \longrightarrow {\cal W}^{(1)}_3 \hbox{ strongly in } H^1(0,L),$$

$\bullet$  ${\cal Q}^{(n)}_3 \in W^{2,\infty}(-1/n,L)$ with ${\cal Q}^{(n)}_3 =0$ in $[-1/n,1/n]$  and
$${\cal Q}^{(n)}_3 \longrightarrow {\cal Q}^{(1)}_3 \hbox{ strongly in } H^1(0,L),$$

$\bullet$  $\overline{u}^{(n)} \in W^{1,\infty}(\Omega;\R^3)$ with $\overline{u}^{(n)}=0$ on $\partial\omega\times ]-1,1[$, $\overline{u}^{(n)}=0$ in the cylinder $D(O, 1/n)\times ]-1,1[$  and
$$\overline{u}^{(n)} \longrightarrow  \overline{\overline{u}}^{(1)}_p \hbox{ strongly in } L^2(\omega ; H^1(-1,1;\R^3)),$$

$\bullet$  $\overline{w}^{(n)} \in W^{1,\infty}(]-1/n,L[\times D;\R^3)$ with $\overline{w}^{(n)}=0$  in the cylinder $D\times ]-1/n,1/n[$  and
$$\overline{w}^{(n)} \longrightarrow \overline{\overline{w}}^{(1)}_r \hbox{ strongly in } L^2(0,L ; H^1(D;\R^3)).$$

First, the above strong convergences and the expression of ${\cal J}$ show that 
\begin{equation}\label{stepn}
\lim_{n\to +\infty} {\cal J}_3({\cal U}^{(n)},{\cal W}^{(n)},{\cal Q}^{(n)}_3)={\cal J}_3({\cal U}^{(1)},{\cal W}^{(1)},{\cal Q}^{(1)}_3).
\end{equation}

For $n$ fixed, let us consider the following sequence $(v_\delta)$ of deformations of the whole structure ${\cal S}_{\delta,\eps}$, defined below:

$\bullet$ in $\Omega_\delta$ we set
\begin{equation}\label{testnP}
\begin{aligned}
v_{\delta,1}(x)&=x_1+\delta^{\kappa-1}\big({\cal U}^{(n)}_1(x_1,x_2)-{x_3\over \delta}{\partial {\cal U}^{(n)}_3\over \partial x_1}(x_1,x_2)+\delta\overline{u}^{(n)}_{1}(x_1,x_2,{x_3\over \delta}\big)\big),\\
v_{\delta,2}(x)&=x_2+\delta^{\kappa-1}\big({\cal U}^{(n)}_2(x_1,x_2)-{x_3\over \delta}{\partial {\cal U}^{(n)}_3\over \partial x_2}(x_1,x_2)+\delta\overline{u}^{(n)}_{2}(x_1,x_2,{x_3\over \delta}\big)\big),\\
v_{\delta,3}(x)&=x_3+\delta^{\kappa-2}\big({\cal U}^{(n)}_3(x_1,x_2)+\delta^2\overline{u}^{(n)}_{3}(x_1,x_2,{x_3\over \delta}\big)\big).
\end{aligned}
\end{equation} 

$\bullet$ in $B_{\eps,\delta}$ we set
\begin{equation}\label{testnR}
\begin{aligned}
v_{\delta,1}(x)&=x_1+\delta^{\kappa-1}\big({\cal U}^{(n)}_1(x_1,x_2)-{x_3\over \delta}{\partial {\cal U}^{(n)}_3\over \partial x_1}(x_1,x_2)\big)+\eps^{\kappa^{'}-2}\big({\cal W}^{(n)}_1(x_3)\\
&\enskip -x_2 {\cal Q}^{(n)}_3(x_3)+\eps^2\overline{w}^{(n)}_{1}\big({x_1\over \eps},{x_2\over \eps},x_3\big)\big),\\
v_{\delta,2}(x)&=x_2+\delta^{\kappa-1}\big({\cal U}^{(n)}_2(x_1,x_2)-{x_3\over \delta}{\partial {\cal U}^{(n)}_3\over \partial x_2}(x_1,x_2)\big)+\eps^{\kappa^{'}-2}\big({\cal W}^{(n)}_2(x_3)\\
&\enskip +x_1  {\cal Q}^{(n)}_3(x_3)+\eps^2\overline{w}^{(n)}_{2}\big({x_1\over \eps},{x_2\over \eps},x_3\big)\big),\\
v_{\delta,3}(x)&=x_3+\delta^{\kappa-2}{\cal U}^{(n)}_3(x_1,x_2)+\eps^{\kappa^{'}-1}\big(\big[{\cal W}^{(n)}_3(x_3)-{\cal U}^{(n)}_3(0,0)\big]-{x_1\over \eps}{d {\cal W}^{(n)}_1\over dx_3}(x_3)\\
&\enskip -{x_2\over \eps}{d {\cal W}^{(n)}_2\over dx_3}(x_3)+\eps\overline{w}^{(n)}_{3}\big({x_1\over \eps},{x_2\over \eps},x_3\big)\big).
\end{aligned}
\end{equation} Obviously, if $\delta$ is small enough (in order to have $\delta\le 1/n$) the two expressions of $v_\delta$ 
match in  the cylinder $C_{\delta,\eps}$ and are equal to 
\begin{equation}\label{testnCyl}
\begin{aligned}
v_{\delta,1}(x)&=x_1+\delta^{\kappa-1}\big({\cal U}^{(n)}_1(x_1,x_2)-{x_3\over \delta}{\partial {\cal U}^{(n)}_3\over \partial x_1}(x_1,x_2)\big),\\
v_{\delta,2}(x)&=x_2+\delta^{\kappa-1}\big({\cal U}^{(n)}_2(x_1,x_2)-{x_3\over \delta}{\partial {\cal U}^{(n)}_3\over \partial x_2}(x_1,x_2)\big),\\
v_{\delta,3}(x)&=x_3+\delta^{\kappa-2} {\cal U}^{(n)}_3(x_1,x_2).
\end{aligned}
\end{equation} 
By construction the deformation $v_\delta$ belongs to $\D_{\delta,\eps}$. Then we have
\begin{equation}\label{630}
m_\delta\le J_\delta(v_\delta).
\end{equation}

In the expression \eqref{testnP} of the displacement  $v_\delta-I_d$ the explicit dependence with respect to $\delta$ permits to derive directly the limit of the Green-St Venant's strain tensor as $\delta$ tends to 0 ($n$ being fixed)
\begin{equation}\label{631}
{1\over 2\delta^{\kappa-1}}\Pi _\delta \big((\nabla_xv_\delta)^T\nabla_x v_\delta -\GI_3\big)\longrightarrow
\GE^{(n)}_p\qquad\hbox{strongly in}\quad L^\infty(\Omega;\R^9),
\end{equation} where the symmetric matrix $\GE^{(n)}_p$ is defined by
\begin{equation*}
\GE^{(n)}_p=\begin{pmatrix}
\displaystyle  -X_3{\partial^2{\cal U}^{(n)}_3\over \partial x_1^2}+{\cal Z}^{(n)}_{11} & \displaystyle  -X_3 {\partial^2 {\cal U}^{(n)}_3\over \partial x_1\partial x_2}+{\cal Z}^{(n)}_{12}
&\displaystyle  {1\over 2}{\partial\overline{u}^{(n)}_{1} \over \partial X_3}\\
* & \displaystyle  -X_3{\partial^2 {\cal U}^{(n)}_3\over \partial x_2^2}+{\cal Z}^{(n)}_{22}  &\displaystyle  {1\over 2}{\partial\overline{u}^{(n)}_{2} \over \partial X_3}\\
* & *&  \displaystyle  {\partial\overline{u}^{(n)}_{3} \over \partial X_3}
\end{pmatrix}\end{equation*} 

\noindent Now, in the rod  $B_{\eps,\delta}$ we have
\begin{equation}\label{testnR2}
\begin{aligned}
v_{\delta,1}(x)&=x_1+\eps^{\kappa^{'}-2}\Big[{\cal W}^{(n)}_1(x_3)+\delta\eps{\cal U}^{(n)}_1(0,0)- \eps x_3 {\partial {\cal U}^{(n)}_3\over \partial x_1}(0,0)\\
&\qquad\quad    -x_2 {\cal Q}^{(n)}_3(x_3)\Big]+\widetilde{w}^{(n)}_{\eps,1}(x),\\
v_{\delta,2}(x)&=x_2+\eps^{\kappa^{'}-2}\Big[{\cal W}^{(n)}_2(x_3)+\delta\eps{\cal U}^{(n)}_2(0,0) -\eps x_3{\partial {\cal U}^{(n)}_3\over \partial x_2}(0,0)\\
&\qquad\quad  +x_1 {\cal Q}^{(n)}_3(x_3)\Big]+\widetilde{w}^{(n)}_{\eps,2}(x),\\
v_{\delta,3}(x)&=x_3+\eps^{\kappa^{'}-1}\big[ {\cal W}^{(n)}_3(x_3) -{x_1\over \eps}{d {\cal W}^{(n)}_1\over dx_3}(x_3)+x_1 {\partial {\cal U}^{(n)}_3\over \partial x_1}(0,0)\\
&\qquad\quad -{x_2\over \eps}{d {\cal W}^{(n)}_2\over dx_3}(x_3)+x_2{\partial {\cal U}^{(n)}_3\over \partial x_2}(0,0)\Big]+\widetilde{w}^{(n)}_{\eps,3}(x).
\end{aligned}
\end{equation} where
\begin{equation*}
\begin{aligned}
\widetilde{w}^{(n)}_{\eps,1}(x)&=\eps^{\kappa^{'}}\overline{w}^{(n)}_{1}\big({x_1\over \eps},{x_2\over \eps},x_3\big)+\delta\eps^{\kappa^{'}-1}({\cal U}^{(n)}_1(x_1,x_2)-{\cal U}^{(n)}_1(0,0))\\
&\qquad - x_3\eps^{\kappa^{'}-1} \Big({\partial {\cal U}^{(n)}_3\over \partial x_1}(x_1,x_2)- {\partial {\cal U}^{(n)}_3\over \partial x_1}(0,0)\Big),\\
\widetilde{w}^{(n)}_{\eps,2}(x)&=\eps^{\kappa^{'}}\overline{w}^{(n)}_{2}\big({x_1\over \eps},{x_2\over \eps},x_3\big)+\delta\eps^{\kappa^{'}-1}({\cal U}^{(n)}_2(x_1,x_2)-{\cal U}^{(n)}_2(0,0))\\
&\qquad - x_3\eps^{\kappa^{'}-1} \Big({\partial {\cal U}^{(n)}_3\over \partial x_2}(x_1,x_2)- {\partial {\cal U}^{(n)}_3\over \partial x_2}(0,0)\Big),\\
\widetilde{w}^{(n)}_{\eps,3}(x)&=\eps^{\kappa^{'}}\overline{w}^{(n)}_{3}\big({x_1\over \eps},{x_2\over \eps},x_3\big)+\eps^{\kappa^{'}-1}\Big({\cal U}^{(n)}_3(x_1,x_2)-{\cal U}^{(n)}_3(0,0)\\
&\qquad -x_1{\partial {\cal U}^{(n)}_3\over \partial x_1}(0,0)-x_2 {\partial {\cal U}^{(n)}_3\over \partial x_2}(0,0)\Big)
\end{aligned}\end{equation*} First notice that
\begin{equation}\label{TwLim}
\begin{aligned}
{1\over \eps^{\kappa^{'}}}P_\eps(&\widetilde{w}^{(n)}_{\eps})  \longrightarrow  \overline{w}^{(n)}_r= \overline{w}^{(n)}-x_3 \Big[X_1{\partial^2 {\cal U}^{(n)}_3\over \partial x_1^2}(0,0)+X_2 {\partial^2 {\cal U}^{(n)}_3\over \partial x_1\partial x_2}(0,0)\Big]\Ge_1\\
&-x_3 \Big[X_1{\partial^2 {\cal U}^{(n)}_3\over \partial x_1\partial x_2}(0,0)+X_2 {\partial^2 {\cal U}^{(n)}_3\over \partial x_2^2}(0,0)\Big]\Ge_2\quad \hbox{strongly in }\; W^{1,\infty}(B;\R^3).
\end{aligned}\end{equation} 
As above,  the expression \eqref{testnR2} of the displacement  $v_\delta-I_d$ being  explicit   with respect to $\delta$ and $\eps$, a   direct calculation gives
\begin{equation}\label{633}
{1\over 2\eps^{\kappa^{'}-1}}P_\eps \big((\nabla v_\delta)^T\nabla v_\delta-\GI_3\big)\longrightarrow \GE^{(n)}_r\qquad\hbox{strongly in}\quad  L^\infty(B ; \R^{3\times 3}),
\end{equation}  where the symmetric matrices $\GE^{(n)}_r$ and $\GF^{(n)}$ are defined by  
\begin{equation}\label{634}
\begin{aligned}
\GE^{(n)}_r&= \begin{pmatrix}
\ds   \gamma_{11}(\overline{w}^{(n)}_r) & \ds   \gamma_{12}(\overline{w}^{(n)}_r) & \ds -{1\over 2}X_2{d{\cal Q}^{(n)}_3\over dx_3}+{1\over 2}{\partial\overline{w}^{(n)}_{r,3}\over \partial X_1}\\  \\
* & \ds   \gamma_{22}(\overline{w}^{(n)}_r) & \ds {1\over 2}X_1{d{\cal Q}^{(n)}_3\over dx_3}+{1\over 2}{\partial\overline{w}^{(n)}_{r,3}\over \partial X_2}\\  \\
* & * &  \ds -X_1{d^2{\cal U}^{(n)}_1\over dx^2_3}-X_2{d^2{\cal U}^{(n)}_2\over dx^2_3}+{d{\cal U}^{(n)}_3\over dx_3} &
\end{pmatrix}+\GF^{(n)},\\
\GF^{(n)}&=\left\{
\begin{aligned}
&{1\over 2}\big(||{\cal Q}^{(n)}||^2_2\GI_3-{\cal Q}^{(n)}.\big({\cal Q}^{(n)}\big)^T\big)\quad \hbox{if} \; \kappa^{'}=3,\\
&0\hskip 3.5cm \hbox{ if } \; \kappa^{'}>3.
\end{aligned}\right.
\end{aligned}
 \end{equation}
 From the strong convergences \eqref{631}-\eqref{633} and taking to account the expressions of the applied forces \eqref{ForceP} and the ones of the deformation, we get
 \begin{equation*}
 \begin{aligned}
 \lim_{\delta\to 0}{1\over \delta^{2\kappa-1}}\int_{{\cal S}_{\delta,\eps}}\widehat{W}_\eps(\nabla  v_\delta)(x)dx&=\int_{\Omega}Q\big(\GE^{(n)}_p\big)+ \int_{B} Q\big(\GE^{(n)}_r\big)\\
 \lim_{\delta\to 0}{1\over \delta^{2\kappa-1}}\int_{{\cal S}_{\delta,\eps}} f_{\delta}\cdot(v_\delta -I_d)&= {\cal L}_3({\cal U}^{(n)},{\cal W}^{(n)},{\cal Q}^{(n)}_3).
\end{aligned} \end{equation*}  Then, from the above limits and \eqref{630} we finally get
 \begin{equation}\label{LimStep2}
 \limsup_{\delta\to0}{m_\delta\over \delta^{2\kappa-1}}\le \int_{\Omega}Q\big(\GE^{(n)}_p\big)+ \int_{B} Q\big(\GE^{(n)}_r\big)-{\cal L}_3({\cal U}^{(n)},{\cal W}^{(n)},{\cal Q}^{(n)}_3).
 \end{equation} Now, $n$ goes to infinity, the above inequality and \eqref{stepn} give
  \begin{equation}\label{LimStep2}
 \limsup_{\delta\to0}{m_\delta\over \delta^{2\kappa-1}}\le {\cal J}_3({\cal U}^{(1)},{\cal W}^{(1)},{\cal Q}^{(1)}_3).
 \end{equation}
This conclude the proof of the theorem.
\end{proof}
\begin{remark} Let us point out that Theorem \ref{theo9.1} shows that for any minimizing sequence $(v_\delta)_\delta$ as in Step 1, the third convergence of the rescaled Green-St Venant's strain tensor in \eqref{E0} is a strong convergence in $L^2(\Omega;\R^{3\times 3})$ and the convergence \eqref{612} is a strong convergence in $L^2(B;\R^{3\times 3})$.
\end{remark}

\section{ Appendix}

\begin{proof}[Proof of Lemma \eqref{lemAp}.] The first  estimate  \eqref{GsPlaque8} is proved in Lemma 4.3 of  \cite{BGJE}).  Now we carry on by estimating $\GG_s( u , B_{\varepsilon,\delta}) $. 
\vskip 1mm
\noindent{\it Step 1.  In this step we prove the following inequality:
\begin{equation}\label{80}
\begin{aligned}
\GG_s( u , B_{\varepsilon,\delta}) \le  &C||\hbox{dist}(\nabla v , SO(3))||_{L^2(B_{\eps,\delta})}\\
 &\qquad +C{||\hbox{dist}(\nabla v , SO(3))||^2_{L^2(B_{\eps,\delta})}\over \eps^3}+ C\eps|||\GQ(0)-\GI_3|||^2.
\end{aligned}\end{equation} }

\noindent The restriction of the displacement $u=v-I_d$ to the rod $B_{\eps,\delta}$ is decomposed as (see Theorem II.2.2 of \cite{BGRod})
\begin{equation}\label{DecRNL8}
u(x)={\cal W} (x_3)+(\GQ(x_3)-\GI_3)\big(x_1\Ge_1+x_2\Ge_2\big)+\overline{w}^{'}(x),\qquad x\in B_{\varepsilon,\delta},
 \end{equation}
where we have ${\cal W}\in H^1(-\delta,L;\R^3) $,  $\GQ\in H^1(-\delta,L; SO(3))$ and $\overline{w}^{'}\in H^1(B_{\varepsilon,\delta};\R^3)$. 
This displacement is also decomposed as in \eqref{DecR}. In both decompositions the field ${\cal W}$ is the average of $u$ on the cross-sections of the rod.
\vskip 1mm
 We know (see Theorem II.2.2  established in \cite{BGRod}) that the fields ${\cal W}$, $\GQ$ and $\overline{w}^{'}$ satisfy
\begin{equation}\label{EstmRod8}
\begin{aligned}
&||\overline{w}^{'}||_{L^2(B_{\varepsilon,\delta};\R^3)}\le C\varepsilon||\hbox{dist}(\nabla v , SO(3))||_{L^2(B_{\eps,\delta})},\\
&||\nabla\overline{w}^{'}||_{L^2(B_{\varepsilon,\delta};\R^{3\times 3})}\le C ||\hbox{dist}(\nabla v , SO(3))||_{L^2(B_{\eps,\delta})}\\
&\Bigl\|{d\GQ\over dx_3}\Big\|_{L^2(-\delta, L ;\R^{3})}\le {C\over\varepsilon^2} ||\hbox{dist}(\nabla v , SO(3))||_{L^2(B_{\eps,\delta})}\\
& \Bigl\|{d{\cal W}\over dx_3}-(\GQ-\GI_3)\Ge_3\Big\|_{L^2(-\delta,L;\R^3)}\le {C\over \varepsilon}||\hbox{dist}(\nabla v , SO(3))||_{L^2(B_{\eps,\delta})}\\
& \bigl\|\nabla v-\GQ\big\|_{L^2(B_{\varepsilon,\delta};\R^{3\times 3})}\le C||\hbox{dist}(\nabla v , SO(3))||_{L^2(B_{\eps,\delta})}
\end{aligned}
\end{equation}
 where the constant $C$ does not depend on $\varepsilon$, $\delta$ and $L$.
 
We set $\Gv=\GQ(0)^T v$  and $\Gu=\Gv-I_d$. The deformation $\Gv$ belongs to $H^1(B_{\eps,\delta};\R^3)$ and satisfies
$$||\hbox{dist}(\nabla \Gv , SO(3))||_{L^2(B_{\eps,\delta})}=||\hbox{dist}(\nabla v , SO(3))||_{L^2(B_{\eps,\delta})}.$$
The last estimate in \eqref{EstmRod8} leads to
\begin{equation}\label{8a}
\begin{aligned}
 \bigl\|\nabla \Gu+(\nabla \Gu)^T\big\|_{L^2(B_{\varepsilon,\delta};\R^{3\times 3})}\le & C||\hbox{dist}(\nabla v , SO(3))||_{L^2(B_{\eps,\delta})}\\
& +C\eps||\GQ(0)^T\GQ +\GQ^T\GQ(0)-2\GI_3||_{L^2(-\delta,L;\R^9)}
\end{aligned} \end{equation} First, we observe that for any matrices $\GR\in SO(3)$ we get $|||\GR-\GI_3|||^2=\sqrt 2|||\GR+\GR^T-2\GI_3|||$. Hence, we have $\sqrt 2|||\GQ(0)^T\GQ +\GQ^T\GQ(0)-2\GI_3|||=|||\GQ-\GQ(0)|||^2$ and using again \eqref{EstmRod8} we obtain
\begin{equation*}
||\GQ(0)^T\GQ +\GQ^T\GQ(0)-2\GI_3||_{L^2(-\delta,L;\R^9)} \le C{||\hbox{dist}(\nabla v , SO(3))||^2_{L^2(B_{\eps,\delta})}\over \eps^4}
\end{equation*} which implies with \eqref{8a}
\begin{equation}\label{8a1}
\GG_s( \Gu , B_{\varepsilon,\delta})\le  C||\hbox{dist}(\nabla v , SO(3))||_{L^2(B_{\eps,\delta})}+C{||\hbox{dist}(\nabla v , SO(3))||^2_{L^2(B_{\eps,\delta})}\over \eps^3}. \end{equation}
Observing that $\nabla u+(\nabla u)^T=\nabla \Gu+(\nabla \Gu)^T+\big(\GI_3-\GQ(0)\big)^T\big(\nabla u-(\GQ(0)-\GI_3)\big)+\big(\nabla u-(\GQ(0)-\GI_3)\big)^T\big(\GI_3-\GQ(0)\big)+2(\GQ(0)+\GQ(0)^T-2\GI_3)$, we deduce that
\begin{equation*}
\begin{aligned}
\GG_s(u , B_{\varepsilon,\delta})& \le  \GG_s( \Gu , B_{\varepsilon,\delta})+2|||\GQ(0)-\GI_3||| \big\|\nabla u-(\GQ(0)-\GI_3)\big\|_{L^2(B_{\varepsilon,\delta};\R^{3\times 3})}\\
&+C\eps|||\GQ(0)+\GQ(0)^T-2\GI_3|||\\
 & \le  \GG_s( \Gu , B_{\varepsilon,\delta}) + C|||\GQ(0)-\GI_3||| {||\hbox{dist}(\nabla v , SO(3))||_{L^2(B_{\eps,\delta})}\over \eps}\\
 &+C\eps|||\GQ(0)-\GI_3|||^2\\
 & \le  \GG_s( \Gu , B_{\varepsilon,\delta}) + C{||\hbox{dist}(\nabla v , SO(3))||^2_{L^2(B_{\eps,\delta})}\over \eps^3}+C\eps|||\GQ(0)-\GI_3|||^2
\end{aligned}\end{equation*} Thanks to \eqref{8a1} we obtain \eqref{80}.
\vskip 1mm
\noindent Now we carry on by giving two estimates on $|||\GQ(0)-\GI_3|||^2$.
\vskip 1mm
\noindent{\it Step 2. First estimate on  $|||\GQ(0)-\GI_3|||^2$.}
\vskip 1mm
We deal with the restriction of $v$ to the plate. Due to Theorem 3.3  established in \cite{BGJE}, the displacement  $u=v-I_d$ is decomposed as 
\begin{equation}\label{FDec8}
u(x)={\cal V} (x_1,x_2)+ x_3(\GR(x_1,x_2)-\GI_3) \Ge_3+\overline{v} (x),\qquad x\in \Omega_\delta
\end{equation}
where ${\cal V} $ belongs to $ H^1(\omega; \R^3)$, $\GR$  belongs to $H^1(\omega; \R^{3\times 3})$ and $\overline{v} $ belongs to $  H^1(\Omega_\delta; \R^3)$ and we have the following estimates
\begin{equation}\label{8.7}
\begin{aligned}
&||\overline{v} ||_{L^2(\Omega_\delta; \R^3)}\le C\delta ||dist(\nabla  v,SO(3))||_{L^2(\Omega_\delta)}\\
&||\nabla \overline{v} ||_{ L^2(\Omega_\delta; \R^9)}\le C ||dist(\nabla  v,SO(3))||_{L^2(\Omega_\delta)}\\
&\Bigl\|{\partial \GR\over \partial x_\alpha}\Big\|_{ L^2(\omega; \R^9)}\le {C\over \delta^{3/2}} ||dist(\nabla  v,SO(3))||_{L^2(\Omega_\delta)}\\
& \Bigl\|{\partial{\cal V}\over \partial x_\alpha}-(\GR-\GI_3) \Ge_\alpha\Big\|_{ L^2(\omega; \R^3)}\le {C\over \delta^{1/2}}||dist(\nabla  v,SO(3))||_{L^2(\Omega_\delta)}\\
& \bigl\|\nabla  v-\GR \big\|_{ L^2(\Omega_\delta; \R^9)}\le C||dist(\nabla  v,SO(3))||_{L^2(\Omega_\delta)}
\end{aligned}
\end{equation} where the constant $C$ does not depend on $\delta$. The following boundary conditions are satisfied
\begin{equation}\label{CLUR8}
{\cal V}=0,\quad  \GR=\GI_3\qquad \hbox{on } \enskip \gamma_0,\qquad \overline{v}=0\qquad \hbox{on}\quad \Gamma_{0,\delta}.
\end{equation}
\vskip 1mm
The last estimates in \eqref{EstmRod8} and \eqref{8.7} allow to compare $\GQ-\GI_3$ and $\GR-\GI_3$ in the cylinder $C_{\delta, \eps}$. We obtain
\begin{equation*}
\begin{aligned}
\eps^2 ||\GQ-\GI_3||^2_{L^2(-\delta,\delta; \R^9)}\le C\big\{||dist(\nabla v, SO(3))||^2_{L^2(\Omega_\delta)}+ ||\hbox{dist}(\nabla v , SO(3))||^2_{L^2(B_{\eps,\delta})}\big\}\\
+C \delta||\GR-\GI_3||^2_{L^2(D_\eps;\R^9)}
\end{aligned}\end{equation*} Besides,   the  third estimate in \eqref{8.7} and the boundary condition on $\GR$ lead to
\begin{equation}\label{R80}
\begin{aligned}
&||\GR-\GI_3||^2_{L^2(D_\eps;\R^9)}\le C\eps^{3/2}||\GR-\GI_3||^2_{L^8(D_\eps;\R^9)}\\
\le &C\eps^{3/2}||\GR-\GI_3||^2_{H^1(D_\eps;\R^9)}\le C\eps^{3/2}{||dist(\nabla  v,SO(3))||^2_{L^2(\Omega_\delta)}\over \delta^3}.
\end{aligned}\end{equation} 
Then, we get
\begin{equation}\label{820}
\begin{aligned}
\eps^2 ||\GQ-\GI_3||^2_{L^2(-\delta,\delta; \R^9)}\le C\big\{||dist(\nabla v, SO(3))||^2_{L^2(\Omega_\delta)}+ ||\hbox{dist}(\nabla v , SO(3))||^2_{L^2(B_{\eps,\delta})}\big\}\\
+C\eps^{3/2}{||dist(\nabla v, SO(3))||^2_{L^2(\Omega_\delta)}\over \delta^{2}}.
\end{aligned}\end{equation}
\noindent Furthermore, the third estimate in \eqref{EstmRod8} gives
\begin{equation*}
\begin{aligned}
|||\GQ(0)-\GI_3|||^2\le {C\over \delta}||\GQ-\GI_3||^2_{L^2(-\delta,\delta; \R^9)}+C\delta\Big\|{d\GQ\over dx_3}\Big\|^2_{L^2(B_{\eps,\delta};\R^9)}\\
\le  {C\over \delta}||\GQ-\GI_3||^2_{L^2(-\delta,\delta; \R^9)}+C{\delta\over \eps^4}||\hbox{dist}(\nabla v , SO(3))||^2_{L^2(B_{\eps,\delta})}
\end{aligned}\end{equation*}
which using \eqref{820} yields
\begin{equation*}
\begin{aligned}
\eps  |||\GQ(0)-\GI_3|||^2\le C\Big[{\delta^{2}\over \eps}+\eps^{1/2} \Big]{||dist(\nabla v, SO(3))||^2_{L^2(\Omega_\delta)}\over  \delta^{3}}\\
+C\Big[\delta+{\eps^2\over \delta}\Big]{ ||\hbox{dist}(\nabla v , SO(3))||^2_{L^2(B_{\eps,\delta})}\over \eps^3}
\end{aligned}\end{equation*} Finally \eqref{80} and the above estimate lead to  
\begin{equation}\label{GGs}
\begin{aligned}
\GG_s( u, B_{\eps,\delta}) \le  C||\hbox{dist}(\nabla v , SO(3))||_{L^2(B_{\eps,\delta})}  & +C\Big[1+{\eps^2\over \delta}\Big]{||\hbox{dist}(\nabla v , SO(3))||^2_{L^2(B_{\eps,\delta})}\over \eps^3}\\
 &\hskip-1cm  + C\big[\delta^{2}+\eps\big]{||dist(\nabla v, SO(3))||^2_{L^2(\Omega_\delta)}\over \eps\delta^{3}}.
\end{aligned}\end{equation}
\vskip 1mm
\noindent{\it Step 3. Second estimate on  $|||\GQ(0)-\GI_3|||^2$.}
\vskip 1mm
Now, we consider the traces of the two decompositions \eqref{DecRNL8} and \eqref{FDec8}
 of  the displacement $u=v-I_d$ on $D_\eps\times\{0\}$. From \eqref{EstmRod8} and \eqref{8.7}  we have
\begin{equation*}
\begin{aligned}
\int_{D_\eps}||u(x_1,x_2,0)-{\cal W}(0)&-(\GQ(0)-\GI_3)(0)  (x_1\Ge_1+x_2\Ge_2)||_2^2\\
&=\int_{D_\eps}||\overline{w}^{'}(x_1,x_2,0)||_2^2\le C\eps||\hbox{dist}(\nabla v , SO(3))||^2_{L^2(B_{\eps,\delta})},\\
\int_{D_\eps}||u(x_1,x_2,0)-{\cal V}(x_1,x_2)||_2^2&=\int_{D_\eps}||\overline{v}(x_1,x_2,0)||_2^2\le C\delta ||\hbox{dist}(\nabla v , SO(3))||^2_{L^2(\Omega_\delta)}.
\end{aligned}\end{equation*} The above estimates  lead to
\begin{equation*}\begin{aligned}
&\int_{D_\eps}||{\cal W}(0)+(\GQ(0)-\GI_3)(x_1\Ge_1+x_2\Ge_2)-{\cal V}(x_1,x_2)||_2^2\\
\le & C\delta ||\hbox{dist}(\nabla v , SO(3))||^2_{L^2(\Omega_\delta)}+C\eps||\hbox{dist}(\nabla v , SO(3))||^2_{L^2(B_{\eps,\delta})}
\end{aligned}\end{equation*}  which implies
\begin{equation}\label{322}
\begin{aligned}
\int_{D_\eps}||(\GQ(0)-\GI_3)(x_1\Ge_1+x_2\Ge_2)-\big({\cal V}(x_1,x_2)-{\cal M}_{D_\eps}({\cal V})\big)||_2^2\\
\le C\delta ||\hbox{dist}(\nabla v , SO(3))||^2_{L^2(\Omega_\delta)}+C\eps||\hbox{dist}(\nabla v , SO(3))||^2_{L^2(B_{\eps,\delta})}.
\end{aligned}\end{equation} 
We carry on by estimating  ${\cal V}-{\cal M}_{D_{\eps}}\big( {\cal V}\big)$.  Let us set
\begin{equation*}
{\GR}_\alpha={\cal M}_{D_\eps}\big((\GR-\GI_3)\Ge_\alpha\big)={1\over |D_\eps|}\int_{D_\eps }(\GR(x_1,x_2)-\GI_3) \Ge_\alpha dx_1dx_2
\end{equation*} and we consider the function $\Phi(x_1,x_2)= {\cal V}(x_1,x_2)-{\cal M}_{D_{\eps}}\big( {\cal V}\big)-x_1\GR_1-x_2\GR_2$. Due to the fourth  estimate in  \eqref{8.7} and the Poincar-Wirtinger's inequality (in order to estimate $||(\GR-\GI_3)\Ge_\alpha-\GR_\alpha||_{L^2(D_\eps; \R^3)}$) we obtain 
\begin{equation}\label{NablaPsi8}
||\nabla\Phi||^2_{L^2(D_\eps ; \R^2)}\le C \Big({1\over \delta}+{\eps^2\over \delta^3}\Big)||\hbox{dist}(\nabla v , SO(3))||^2_{L^2(\Omega_\delta)},\\
\end{equation} Noting that ${\cal M}_{D_\eps}(\Psi)=0$,  the above inequality and the Poincar-Wirtinger's inequality in the disc $D_\eps $ lead to
\begin{equation}\label{Psi8}
||\Phi||^2_{L^2(D_\eps)}\le C{\eps^2\over\delta}\Big(1+{\eps^2\over \delta^2}\Big)||\hbox{dist}(\nabla v , SO(3))||^2_{L^2(\Omega_\delta)}.
\end{equation} 
 Estimates  \eqref{322}  gives
\begin{equation*}
\begin{aligned}
&\int_{D_\eps} ||(\GQ(0)-\GI_3)(x_1\Ge_1+x_2\Ge_2)||^2_2 \le  C\big( ||\Phi||^2_{L^2(D_\eps)}\\
&\quad +\eps^4||{\GR}_1||^2_2+\eps^4||{\GR}_2||^2_2+\delta ||\hbox{dist}(\nabla v , SO(3))||^2_{L^2(\Omega_\delta)}+\eps||\hbox{dist}(\nabla v , SO(3))||^2_{L^2(B_{\eps,\delta})}\big)
\end{aligned}\end{equation*} which in turns with \eqref{R80} and  \eqref{Psi8}  yield
\begin{equation*}
\begin{aligned}
&\eps^4\big( ||(\GQ(0)-\GI_3)\Ge_1||^2_2+ ||(\GQ(0)-\GI_3)\Ge_2||^2_2\big)\\
\le & C\Big({\eps^2\over\delta}+{\eps^{7/2}\over \delta^3}+\delta\Big)||\hbox{dist}(\nabla v , SO(3))||^2_{L^2(\Omega_\delta)}+C\eps||\hbox{dist}(\nabla v , SO(3))||^2_{L^2(B_{\eps,\delta})}
\end{aligned}\end{equation*} and finally
\begin{equation}\label{326}
\begin{aligned}
&\eps|||\GQ(0)-\GI_3|||^2\\
\le & C\Big({\delta^2\over\eps^2}+{1\over \eps^{1/2}}+{\delta^4\over \eps^3}\Big){||\hbox{dist}(\nabla v , SO(3))||^2_{L^2(\Omega_\delta)}\over \delta^{3}}+C\eps{||\hbox{dist}(\nabla v , SO(3))||^2_{L^2(B_{\eps,\delta})}\over \eps^3}.
\end{aligned}\end{equation} Estimates \eqref{80} and \eqref{326} yield 
\begin{equation}\label{GGs1}
\begin{aligned}
\GG_s( u, B_{\eps,\delta}) \le  C||\hbox{dist}(\nabla v , SO(3))||_{L^2(B_{\eps,\delta})}  & +C{||\hbox{dist}(\nabla v , SO(3))||^2_{L^2(B_{\eps,\delta})}\over \eps^3}\\
 &\hskip-2cm  + C\Big[\eps^{1/2}+{\delta^2\over \eps}+{\delta^4\over \eps^2}\Big]{||dist(\nabla v, SO(3))||^2_{L^2(\Omega_\delta)}\over \eps\delta^{3}}.
\end{aligned}\end{equation}  
\vskip 1mm
\noindent{\it Step 4. Final estimate on $\GG_s( u, B_{\eps,\delta})$.}
\vskip 1mm
\noindent The two estimates of $\GG_s( u, B_{\eps,\delta}) $ given by \eqref{GGs} and \eqref{GGs1} lead to
\vskip 1mm
\noindent $\bullet$ if $\eps^2\le \delta$ then 
\begin{equation*}
\begin{aligned}
\GG_s( u, B_{\eps,\delta}) \le  C||\hbox{dist}(\nabla v , SO(3))||_{L^2(B_{\eps,\delta})}  & +C{||\hbox{dist}(\nabla v , SO(3))||^2_{L^2(B_{\eps,\delta})}\over \eps^3}\\
 &\hskip-1cm  + C\big[\delta+\eps\big]{||dist(\nabla v, SO(3))||^2_{L^2(\Omega_\delta)}\over \eps\delta^{3}}.
\end{aligned}\end{equation*} $\bullet$  if $\delta\le \eps^2$ then 
\begin{equation*}
\begin{aligned}
\GG_s( u, B_{\eps,\delta}) \le  C||\hbox{dist}(\nabla v , SO(3))||_{L^2(B_{\eps,\delta})}  & +C{||\hbox{dist}(\nabla v , SO(3))||^2_{L^2(B_{\eps,\delta})}\over \eps^3}\\
 &\hskip-0.3cm  + C\eps^{1/2} {||dist(\nabla v, SO(3))||^2_{L^2(\Omega_\delta)}\over \eps\delta^{3}}.
\end{aligned}\end{equation*}   We immediately deduce \eqref{GGs11}.
\end{proof}

\end{document}